\documentclass[reqno,oneside]{amsart}
\usepackage{amsmath}
\usepackage{amssymb,enumitem,multicol}
\usepackage{amsthm}
\usepackage[foot]{amsaddr}
\usepackage{mathrsfs}
\usepackage{hyperref}
\usepackage{tikz}
\usepackage{tikz-3dplot}
\hypersetup{
    colorlinks=true,
    linkcolor=black,
    citecolor=black
    }
\newtheorem{theorem}{Theorem}
\newtheorem{proposition}[theorem]{Proposition}

\newtheorem{conjecture}[theorem]{Conjecture}
\newtheorem{lemma}[theorem]{Lemma}
\theoremstyle{remark}
\newtheorem{definition}[theorem]{Definition}
\newtheorem{remark}[theorem]{Remark}
\newtheorem{example}[theorem]{Example}

\newcommand{\R}{\mathbb{R}}
\newcommand{\C}{\mathbb{C}}
\newcommand{\Z}{\mathbb{Z}}

\begin{document}

\title[Equiangular Tight Frames in real symplectic space]{\centering Equiangular tight frames in real symplectic space: Zauner's conjecture and the skew Hadamard conjecture}
\author{Kean~Fallon}
\address{Department of Mathematics, Iowa State University, Ames, IA}
\email{keanpf@iastate.edu}

\begin{abstract}
We introduce the notion of equiangular tight frames in real symplectic spaces and formulate a conjecture on their existence in terms of the dimension and number of vectors.
Our main results shows the \textit{symplectic Zauner's conjecture} is equivalent to the skew Hadamard conjecture. 
The proof involves counting subgraphs called \emph{diamonds} in tournaments.   
\end{abstract}

\maketitle

\section{Introduction}
Suppose that we are tasked with arranging at least $d$ lines in $\mathbb{F}^d$, where $\mathbb{F}$ is $\R$ or $\C$. 
An optimal packing of these lines is one in which the lines are as spread apart as possible.
Given unit-norm representatives $\{\psi_i\}_{i=1}^n$ of $n\geq d$ lines in $\mathbb{F}^d$, the \textit{coherence} of $\{\psi_i\}_{i=1}^n$ is $\mu:=\displaystyle\max_{i\neq j} |\langle \psi_i,\psi_j\rangle|$, where $\langle \cdot, \cdot\rangle$ is the standard Euclidean or Hermitian inner product. 
An optimal line packing is thus one that minimizes coherence.

Identifying $\Psi\in \mathbb{F}^{d\times n}$ as the matrix with $i$th column $\psi_i$, a close look at the inequality
\begin{equation*} 
0 \leq \left\|\Psi\Psi^* - \frac{n}{d}I\right\|_F =\sum_{i,j=1}^n |\langle{\psi_i,\psi_j}\rangle|^2 - \frac{n^2}{d}\leq (n^2-n)\mu^2 + n-\frac{n^2}{d} 
\end{equation*}
shows that 
\begin{equation*}
\mu \geq\sqrt{\dfrac{n-d}{d(n-1)}}.
\end{equation*}
This is known as the Welch bound~\cite{welch2003lower} and it is saturated if and only if $\Psi$ satisfies
\begin{enumerate}
    \item[(i)] (\emph{Equiangularity}) $|\langle\psi_i,\psi_j\rangle| = \mu$ for all $i\neq j$, and
    \item[(ii)] (\emph{Tightness}) $\Psi\Psi^* = \frac{n}{d}I$.
\end{enumerate}

If $\Psi$ satisfies (i) and (ii), then $\Psi$ is called a $d\times n$ \emph{equiangular tight frame} (or \emph{ETF}). 
In this paper, we define ETFs in real symplectic vector spaces, that is, real vector spaces equipped with a symplectic form. 
These ETFs are similar in spirit to the above, but they are not the same object (see Definitions \ref{def:sympframe}, \ref{def:tight}, \ref{def:equiangular} for a preview if desired).
In particular, our work ties together two major open problems: Zauner's conjecture and the Hadamard conjecture.

\emph{Gerzon's bound}~\cite{lemmens1973equiangular} states that if $\Psi\in \mathbb{F}^{d\times n}$ is an ETF, then 
\[
n\leq \begin{cases} \frac{1}{2}d(d+1) & \text{if }\mathbb{F} = \R, \\
d^2 & \text{if }\mathbb{F} = \C.
\end{cases}
\]
While equality does not hold for arbitrary $d$ when $\mathbb{F} = \R$ (for instance, see~\cite{sustik2007existence}), the same cannot be said when $\mathbb{F}=\C$.
In his dissertation~\cite{zauner1999quantum}, Gerhard Zauner conjectured that there exists a $d\times d^2$ complex ETF for every dimension $d$. 
This conjecture, colloquially named \emph{Zauner's conjecture}, has generated considerable interest since its proposal in 1999.
These maximal complex ETFs are symmetric informationally-complete positive operator-valued measures (SIC-POVMs), and have garnered attention in quantum information theory~\cite{chen2015general,fuchs2017sic,renes2004symmetric,scott2006tight,shang2018enhanced,zauner1999quantum}, compressed sensing~\cite{bandeira2013road, herman2009high}, phase retrieval~\cite{fannjiang2020numerics}, and design theory~\cite{Graydon2015QuantumCD}. 
Zauner's conjecture may also be related to Hilbert's 12th problem~\cite{appleby2025constructiveapproachzaunersconjecture}.

Another famous conjecture that bares similar characteristics to Zauner's conjecture is the \emph{Hadamard conjecture}, which concerns itself with the existence of \emph{Hadamard matrices}.
A \textbf{Hadamard matrix of order $n$} is an $n\times n$ matrix $H$ satisfying $H_{ij}\in \{\pm1\}$ and $HH^\top= nI$.
A \textbf{skew Hadamard matrix of order $n$} is a Hadamard matrix $H$ that can be written $H=C+I$, where $C$ is skew symmetric.
Hadamard matrices are used in quantum computing, have applications in compressed sensing, and are closely tied to $D$-optimal design theory~\cite{barenco1995elementary, leung2002simulation, hedayat1978hadamard,pratt1969hadamard,colbourn2010crc}. 
In particular, Hadamard matrices are solutions to determinantal optimization problems~\cite{hadamard1893resolution}.
A Hadamard matrix of order $n$ exists only if $n=1,2$ or if $n$ is a multiple of $4$.
Often attributed to Paley in 1933~\cite{Paley1933OnOM} and possibly conjectured implicitly as early as 1867, the Hadamard conjecture proposes the converse: that a Hadamard matrix of order $n$ exists whenever $n=1,2$ or whenever $n$ is a multiple of $4$. 
Later in~\cite{wallis1971some}, a stronger version of the Hadamard conjecture called the \emph{skew Hadamard conjecture} was proposed, which hypothesized that skew Hadamard matrices of order $n$ exist for the same $n$. 

The Hadamard conjecture(s) and Zauner's conjecture are alike: both are statements on the existence of solutions to optimization problems, both have been resilient in the face of attempts to fully resolve them.
Several constructions have been used to generate families of Hadamard and skew Hadamard matrices in infinitely many dimensions (see ~\cite{koukouvinos2008skew,Tressler2004ASO} for a survey), but there are many orders for which the existence of a Hadamard matrix is unknown.
Currently, the smallest $n$ for which a Hadamard matrix of order $n$ is not known to exist is $n=668$. 
This number has not changed in twenty years, since the publication of~\cite{kharaghani2005hadamard}.
Additionally, most modern results rely on computer searches.
For Zauner's conjecture, progress towards determining the existence of $d\times d^2$ complex ETFs has been difficult.
With the current state of affairs, only finitely many dimensions are known for which there exists a $d\times d^2$ complex ETF, while numerical evidence has been given to support Zauner's conjecture in dimensions as high as 39,604~\cite{appleby2019tight, appleby2022sic, fuchs2017sic,grassl2017fibonacci, renes2004symmetric, scott2010symmetric, zauner1999quantum}.
The second main result of this paper reveals a connection between these two conjectures.

While ETFs are usually studied over complex or real Euclidean vector spaces, recent work has been done in studying ETFs after varying the underlying form~\cite{MR4339582,MR4364998}.
In a similar move, we study frames over real symplectic vector spaces.
The basic components of these frames are quite similar to their complex and real Euclidean cousins: they are a sequence of vectors satisfying certain properties and we can talk about their size.

We define symplectic notions of equiangularity and tightness and in doing so, come to analogs of Gerzon's bound and Zauner's conjecture over symplectic spaces.
The symplectic version of Gerzon's bound (Proposition \ref{prop:gerzon}) is quite restrictive: given a $d\times n$ equiangular tight frame in a real symplectic vector space of dimension $d$, we must have $n\leq d+1$. 
In fact, more can be said. 
We leave this to the following theorem, which is our first main result.

\begin{theorem}[Main Result A]
\label{thm:mainA}
    There exists a $d\times n$ ETF in real symplectic space only if 
\[
n = \begin{cases}
    d & \text{if }d\equiv 0 \text{ mod }4 \text{ or } d = 2, \\
    d+1 & \text{if }d\equiv 2 \text{ mod }4.
    \end{cases}
\]
\end{theorem}

Theorem \ref{thm:mainA} suggests the following symplectic analog of Zauner's conjecture.

\begin{conjecture}[Symplectic Zauner's Conjecture]
\label{conj:symzauners}
There exists a $d\times n$ ETF in real symplectic space if and only if 
\[
n = \begin{cases}
    d & \text{if }d\equiv 0 \text{ mod }4 \text{ or } d = 2, \\
    d+1 & \text{if }d\equiv 2 \text{ mod }4.
    \end{cases}
\]
\end{conjecture}

As we said, the second main result of this paper connects the Hadamard conjecture, specifically the skew Hadamard conjecture, and symplectic Zauner's conjecture.
We emphasize the strength of this connection: when Zauner's conjecture is shifted to the symplectic setting, it is the same thing as the skew Hadamard conjecture. 
The following theorem elaborates.
 
\begin{theorem}[Main Result B]
\label{thm:equiv}
    Symplectic Zauner's conjecture is equivalent to the skew Hadamard conjecture. Specifically, the following hold when $d>1$.
    \begin{itemize}
        \item[(a)] There exists a $d\times d$ ETF in real symplectic space if and only if there exists a skew Hadamard matrix of order $d$.
        \item[(b)] There exists a $d\times (d+1)$ ETF in real symplectic space if and only if there exists a skew Hadamard matrix of order $d+2$.
    \end{itemize}
\end{theorem}

Theorem \ref{thm:equiv} links together two wide open and difficult problems.
As we will see in Section 5, this theorem also plays a part in associating ETFs in real symplectic space with certain complex ETFs.
Meanwhile, the proof of Theorem \ref{thm:equiv} reveals that the underlying combinatorial object behind ETFs in real symplectic space is a special kind of tournament.

The following section gives a brief overview of the basic symplectic linear algebra used in this paper, in particular a useful description of the adjoint of a frame in the symplectic setting. 
Section 3 introduces frames in real symplectic vector spaces and their properties. 
Section 4 defines the notion of equiangularity in this setting before introducing graph theoretical results that are necessary to prove the two theorems above. 
The section culminates in their proof.
Section 5 offers a straightforward construction of $d\times d$ ETFs in real symplectic space, while Section 6 details a relationship between ETFs in real symplectic space and complex ETFs with certain parameters.

\section{Preliminaries}

Let $W$ be a finite-dimensional real or complex vector space. 
A \emph{symplectic form} on $W$ is a map $\omega : W \times W \to \mathbb{R}$ that satisfies 
\begin{enumerate}[itemsep = 1mm]
    \item[(i)] (\emph{Non-degenerate}) Given $x\in W$, one has $\omega(x,y) = 0$ for all $y\in W$ if and only if $x=0$, 
    \item[(ii)] (\emph{Bilinear}) $\omega$ is $\mathbb{R}$-linear in both components, and
    \item[(iii)] (\emph{Alternating}) $\omega(x,x) = 0$ for all $x,y\in W$.
\end{enumerate}

When $W$ is equipped with a symplectic form, it is called a \emph{symplectic vector space}. 
A \emph{symplectic transformation} $M$ on a symplectic vector space $W$ is a linear transformation that preserves the symplectic form.
As a note, a bilinear form $\omega$ on $W$ is alternating if and only if it is \emph{skew symmetric}, that is $\omega(x,y) = -\omega(y,x)$ for all $x,y\in W$. 

\begin{example}\label{ex:complexsympspace}
Consider $\C^d$ with the Hermitian inner product. Writing
\[
\langle{x,y}\rangle_{\C^d} = \operatorname{Re}\langle{x,y}\rangle_{\C^d} + i\operatorname{Im}\langle{x,y}\rangle_{\C^d},
\]
where $\operatorname{Re}$ and $\operatorname{Im}$ are the real and imaginary parts, respectively, one observes that $\operatorname{Re}\langle{\cdot,\cdot}\rangle_{\C^d}$ is a real symmetric inner product while $\operatorname{Im}\langle{\cdot,\cdot}\rangle_{\C^d}$ is a symplectic form. 
To see that $\operatorname{Im}\langle{\cdot,\cdot}\rangle_{\C^d}$ is non-degenerate, notice that for any $x\in \C^d$ one has $\operatorname{Im}\langle{x,ix}\rangle_{\C^d}=0$ if and only if $x = 0$. 
Meanwhile, it is clear that $\operatorname{Im}\langle{\cdot,\cdot}\rangle_{\C^d}$ is bilinear and alternating.
\end{example}

\begin{example}\label{ex:mainsympspace}
When $d$ is even, the form on $\R^d$ given by $[x,y] = x^\top \Omega y$, where
\[
\Omega = \bigoplus_{i=1}^{d/2} \left[\begin{array}{rr} 0 & 1 \\ -1 & 0 \end{array}\right],
\]
is symplectic.
\end{example} 

All symplectic forms on a given vector space are isomorphic, so we may choose to work with one that is convenient for us.
The real symplectic vector space in Example \ref{ex:mainsympspace} is our choice and thus, denoting it by $\R_\mathcal{S}^d$, will be the one we work in over the course of this paper.
The matrix $\Omega$ satisfies $\Omega^\top = -\Omega = \Omega^{-1}$, which is a property we will use often.
Symplectic transformations $M$ on $\R_\mathcal{S}^d$ are those that satisfy $M^\top\Omega M = \Omega$.

By convention, whenever we refer to \emph{real symplectic space}, we mean explicitly $\R_\mathcal{S}^d$. 
When working over $\R_\mathcal{S}^d$, we will occasionally refer to this as work done in the \emph{symplectic setting}.
Moving forward, the usual real vector space of dimension $d$ equipped with the Euclidean inner product $\langle x, y\rangle = x^\top y$ will be denoted by $\R^d_\mathcal{E}$. 
To compare and contrast with results in $\R_\mathcal{S}^d$, we will often refer to working over $\C^d$ or $\R_\mathcal{E}^d$ as working in the complex or real Euclidean setting.

Of particular notational importance, when we say a frame in $\R_\mathcal{S}^d$ is \textit{equiangular}, \textit{tight}, or both (an ETF), we mean so in the context of the symplectic setting unless explicitly stated otherwise.
For example, if we call a sequence in $\R_\mathcal{S}^d$ an equiangular tight frame, we mean that it explicitly satisfies Definitions \ref{def:sympframe}, \ref{def:tight}, and \ref{def:equiangular}, and that this should be clear from the context of residing in $\R_\mathcal{S}^d$.   
In this paper, it becomes crucial for us to define the adjoint of maps between $\R_\mathcal{S}^d$ and $\R_\mathcal{E}^d$. 
To that end, we provide the following, a proof of which can be found in the appendix.
\begin{proposition}
\label{prop:adjoint}
    Let $V,W$ be finite-dimensional real vector spaces equipped with non-degenerate bilinear forms $(\cdot,\cdot)_V$ and $(\cdot,\cdot)_W$, respectively. 
    Choose a linear map $A:V\to W$.
    \begin{enumerate}
        \item[(a)] There exists a unique map $A^\dagger:W\to V$ called the \textbf{adjoint} of $A$ such that 
        \[
        (Av,w)_W = (v,A^\dagger w)_V,
        \]
        \item[(b)] If $B:W\to V$ is a linear map, then $(AB)^\dagger = B^\dagger A^\dagger$,
        \item[(c)] If $A$ is invertible, then $(A^\dagger)^{-1} = (A^{-1})^\dagger$,
        \item[(d)] If the forms $(\cdot,\cdot)_V$ and $(\cdot,\cdot)_W$ are given by $(v_1,v_2)_V =v_1^\top Q_vv_2$ and $(w_1,w_2)_W = w_1^\top Q_w w_2$, where the matrices $Q_v$ and $Q_w$ are invertible and either symmetric or skew symmetric, then 
        \[
        A^\dagger = Q_v^{-1}A^\top Q_w,
        \]
        \item[(e)]If exactly one of $V$ or $W$ is $\R_{\mathcal{S}}^d$ and the other is $\R_\mathcal{E}^n$, then $(A^\dagger)^\dagger = -A$.
    \end{enumerate}
\end{proposition}

Moving forward, we will denote the adjoint of $A$ as $A^\dagger$, whereas $A^*$ is the conjugate-transpose of $A$ and $A^\top$ is the transpose of $A$.

\section{Frames in Real Symplectic Space}

In this section, we define frames in real symplectic space and give their properties. 
Classically, a $d\times n$ frame in the complex or real Euclidean setting is defined as a sequence of vectors $\Psi=\{\psi_i\}_{i=1}^n$ satisfying 
\[
A \|x\|^2 \leq \sum_{i=1}^n |\langle x,\psi_i\rangle|^2 \leq B\|x\|^2
\]
for all $x$ in the appropriate $d$-dimensional vector space, where $A$ and $B$ are positive constants called \emph{frame bounds}. 

Unfortunately, we cannot port the definition given above to the symplectic setting directly, since a symplectic form does not give rise to a norm. 
Instead, the above can be viewed as the condition that the frame operator $\Psi\Psi^\dagger$ is invertible.
Equivalently, there exist positive constants $A,B$ so that 
\begin{equation}\label{eq:frameIneq}
A \leq \lambda \leq B \qquad \text{ for all }\lambda\in \sigma(\Psi\Psi^\dagger).
\end{equation}
This perspective allows for an analogous definition of a frame in $\R_\mathcal{S}^d$ with only a slight adjustment, which we explain afterwards.

\begin{definition} \label{def:sympframe}
    A \textbf{frame} in $\R_\mathcal{S}^d$ is a sequence of vectors $\Phi = \{\varphi_i\}_{i=1}^n$ in $\R_{\mathcal{S}}^d$  together with positive constants $A$ and $B$ such that the following holds:
    \[A \leq |\lambda| \leq B \qquad \text{ for all }\lambda\in \sigma(\Phi\Phi^\dagger).\]
    The quantities $A$ and $B$ are called \emph{lower} and \emph{upper frame bounds}, respectively.
\end{definition}

Let $\Phi = \{\varphi_i\}_{i=1}^n$ be a frame in $\R_\mathcal{S}^d$.
As in the complex or real Euclidean setting, we often identify $\Phi$ with the $d\times n$ matrix $\Phi$ whose columns are the vectors $\varphi_i$, hence the terminology ``$d\times n$ frame". 
We also identify $\Phi$ with the corresponding linear transformation $\Phi:\R_\mathcal{E}^n\to\R_\mathcal{S}^d$, which we call the \emph{synthesis operator}. 
In other words, we have 
\[
\Phi(x) := \Phi x=\sum_{i=1}^n x_i\varphi_i \qquad \text{for all }x=\{x_i\}_{i=1}^n\in \R_\mathcal{E}^n.
\]
The adjoint of the synthesis operator is called the \emph{analysis operator}.
From Proposition $\ref{prop:adjoint}$, we see that $\Phi^\dagger:\R_\mathcal{S}^d \to \R_\mathcal{E}^n$ is equal to $\Phi^\top \Omega$, and so 
\[
\Phi^\dagger(y) := \Phi^\dagger y= \Phi^\top \Omega y = \{[\varphi_i,y]\}_{i=1}^n \qquad \text{for all }y\in\R_\mathcal{S}^d.
\]
We call the operator $\Phi\Phi^\dagger$ the \emph{frame operator} of $\Phi$, while the matrix $\Phi^\dagger\Phi$ is the \emph{Gram matrix} of $\Phi$, that is, the matrix that has $(\Phi^\dagger\Phi)_{ij}=[\varphi_i,\varphi_j]$. 
One can see that the Gram matrix of $\Phi$ is real and skew symmetric. 
This is why the absolute values are important in Definition \ref{def:sympframe}: the nonzero eigenvalues of $\Phi^\dagger\Phi$ are purely imaginary, so the eigenvalues of $\Phi\Phi^\dagger$ must be as well. 

The inequality in definition $\ref{def:sympframe}$ is tight when the frame bounds $A$ and $B$ can be chosen to be equal.
This leads to the following.
\begin{definition}\label{def:tight}
    A frame $\Phi$ is \textbf{$c$-tight} if it has frame bounds $A=B=c>0$.
\end{definition}

As a note, a $c$-tight frame $\Phi$ may be called \emph{tight} if the scalar $c$ is not relevant to the discussion.

While Definition $\ref{def:sympframe}$ is written only for the symplectic setting, allowing $\R_\mathcal{E}^d$ or $\C^d$ to replace $\R_\mathcal{S}^d$ in the definition unifies the notion of a frame across symplectic, real Euclidean, and complex vector spaces.
In other words, supplying absolute values to the inequality (\ref{eq:frameIneq}) does not change the definition in the complex or real Euclidean setting while allowing for a direct extension to the symplectic setting.
Likewise, the same definition of tightness provided above applies in the complex or real Euclidean setting.

\subsection{Basic Properties}\,

Several basic properties of frames in the complex or real Euclidean setting carry over to the symplectic setting with minimal modification. 
For example, a frame in the complex or real Euclidean setting can be thought of as simply a spanning set.
The same is true for a frame in $\R_\mathcal{S}^d$.

\begin{proposition}\label{prop:spanning}
    A sequence of vectors $\Phi = \{\varphi_i\}_{i=1}^n$ in $\R_\mathcal{S}^d$ is a frame if and only if $\Phi$ spans $\R_\mathcal{S}^d$.
\end{proposition}

\begin{proof}
Suppose that $\Phi$ is a frame in $\R_\mathcal{S}^d$, so that $\Phi\Phi^\dagger$ is invertible.
Then $\operatorname{rank}(\Phi\Phi^\dagger)=d$ bounds $\operatorname{rank}(\Phi)$ from below, so it must be that $\Phi$ is surjective.

Conversely, if $\Phi$ is surjective then $\Phi^\dagger$ is injective. 
Thus $\ker(\Phi) = \ker(\Phi^\dagger\Phi)$ and an application of rank-nullity yields
\[
d=\operatorname{rank}(\Phi) = \operatorname{rank}(\Phi^\dagger\Phi).
\]
Since $\Phi^\dagger\Phi$ is skew symmetric, it is diagonalizable.
Therefore, the sum of the algebraic multiplicities of its nonzero eigenvalues is $d$, and so it must be for $\Phi\Phi^\dagger$ as well. 
Thus $0\notin\sigma(\Phi\Phi^\dagger)$, or equivalently $\Phi$ is a frame.
\end{proof}

Treating frames in $\R_\mathcal{S}^d$ as spanning sets provides a simple way of identifying them, but doing so also hides important distinctions.
For instance, it does not make explicit their frame bounds, which we care about.
Moreover, when they are viewed only as matrices over $\R$, a frame in $\R_\mathcal{S}^d$ and a frame in $\R_\mathcal{E}^d$ may appear no different.
However, as the following example shows, it is clear that objects such as the frame operator and Gram matrix do not look the same when comparing between frames in these two different spaces.
This should be of no surprise, as both objects include information about the underlying form.

\begin{example}
\label{ex:basic}
    The frame 
    \[
    \Phi = \left[\begin{array}{rrr} 1 & 0 & 0 \\ 0 & 1 & 1 \end{array}\right]
    \]
    in $\R_\mathcal{S}^2$ has analysis operator, frame operator, and Gram matrix
    \[
    \Phi^\dagger = \left[\begin{array}{rr} 0 & 1 \\ -1 & 0 \\ -1 & 0 \end{array}\right], \qquad\Phi\Phi^\dagger = \left[\begin{array}{rr} 0 & 1 \\ -2 & 0 \end{array}\right], \qquad \Phi^\dagger\Phi = \left[\begin{array}{rrr} 0 & 1 & 1  \\ -1 & 0 & 0 \\ -1 & 0 & 0\end{array}\right]. 
    \]
    Note that $\sigma(\Phi\Phi^\dagger) = \{\pm i\sqrt{2}\}$, so $\Phi$ is $\sqrt{2}$-tight.
\end{example}

In the complex or real Euclidean setting, the frame operator and Gram matrix of a frame are both Hermitian.
Meanwhile, in the symplectic setting the Gram matrix is skew symmetric, while the example above shows that the frame operator is not necessarily symmetric, skew symmetric, or even normal.
However, the frame operator of a frame in real symplectic space is always diagonalizable.

\begin{proposition}
\label{thm:diag}
    If $\Phi$ is a frame in $\R_\mathcal{S}^d$, then $\Phi\Phi^\dagger$ is diagonalizable.
\end{proposition}

\begin{proof}
Let $\Phi$ be a $d\times n$ frame in $\R_\mathcal{S}^d$. 
Denote the distinct nonzero eigenvalues of $\Phi^\dagger\Phi$ by $\lambda_1,\dots,\lambda_r$. 
Since $\Phi^\dagger \Phi$ is skew symmetric, it is diagonalizable, hence its eigenvectors form a basis for its image. 
Thus for each $i=1,\dots,r$ we choose a basis $x_1^{\lambda_i},\dots,x_{m_i}^{\lambda_i}$ for the $\lambda_i$-eigenspace so that the collection of all such vectors forms a basis for the $d$-dimensional image of $\Phi^\dagger\Phi$.
It is straightforward to verify that the collection $$\Phi x_1^{\lambda_1}, \dots ,\Phi x_{m_1}^{\lambda_1}, \dots, \Phi x_1^{\lambda_r}, \dots ,\Phi x_{m_r}^{\lambda_r}$$ is a collection of $d$ eigenvectors for $\Phi\Phi^\dagger$, so it suffices to show that they are linearly independent.

Because $\Phi^\dagger$ is injective, we have
\[
\sum_{i=1}^r\sum_{j=1}^{m_{i}} c_{ij}\Phi x_j^{\lambda_i} = 0 \iff \Phi^\dagger \sum_{i=1}^r\sum_{j=1}^{m_i} c_{ij}\Phi x_j^{\lambda_i} = 0 \iff
\sum_{i=1}^r\lambda_i\sum_{j=1}^{m_i} c_{ij}x_j^{\lambda_i} = 0.
\]
Since the $\lambda_i$ are all nonzero and the $x_j^{\lambda_i}$ are linearly independent, it must be that $c_{ij}=0$ for all $i,j$.
\end{proof}

Frames in the complex and real Euclidean setting are commonly known to provide a reconstruction formula.
Given a frame $\{\psi_i\}_{i=1}^n$ in $\mathbb{F}^d$, where $\mathbb{F} \in \{\R_\mathcal{E},\C\}$, one has 
\[
x = \sum_{i=1}^n\langle \psi_i,x\rangle\psi_i' = \sum_{i=1}^n\langle\psi_i',x\rangle\psi_i \qquad \text{for any }x\in \mathbb{F}^d,
\]
for some \emph{dual frame} $\{\psi_i'\}_{i=1}^n$ of $\{\psi_i\}_{i=1}^n$. 
Frames in real symplectic space offer a similar reconstruction formula.

\begin{proposition}
    Let $\Phi = \{\varphi_i\}_{i=1}^n$ be a sequence of vectors in $\R_{\mathcal{S}}^d$. Then $\Phi$ is a frame if and only if there exists a frame $\{\varphi_i'\}_{i=1}^n$ in $\R_\mathcal{S}^d$ such that
    \[
    x = \sum_{i=1}^n[\varphi_i,x]\varphi_i' = -\sum_{i=1}^n [\varphi_i',x]\varphi_i \qquad \text{for any }x\in \R_\mathcal{S}^d.
    \]
\end{proposition}

\begin{proof}
The reverse direction is an immediate consequence of Proposition \ref{prop:spanning}. 
For the forward direction, let $F = \Phi\Phi^\dagger$ be the frame operator of $\Phi$.
In the first equality, it suffices to find a $d\times n$ matrix $A = \left[\begin{array}{rrr}\varphi_1'&\cdots&\varphi_n'\end{array}\right]$ so that $A\Phi^\dagger = I_d$, the $d\times d$ identity matrix.
This holds when $A = F^{-1}\Phi$, and so we can take $\varphi_i' = F^{-1}\varphi_i$ for each $i=1,\dots,n$.

To see the second equality, first note that $F^\dagger = (\Phi^\dagger)^\dagger\Phi^\dagger = - \Phi\Phi^\dagger = -F$.
It follows that $(F^{-1})^\dagger = -F^{-1}$ and so 
\[
x = F\left(F^{-1}x\right) = \sum_{i=1}^n [\varphi_i,F^{-1}x]\varphi_i = \sum_{i=1}^n[-F^{-1}\varphi_i,x]\varphi_i = -\sum_{i=1}^n [\varphi_i',x]\varphi_i.
\]
\end{proof}

Lastly, we can discern important properties of a frame from its Gram matrix. 
In order for this to be useful, one must be able to retrieve the underlying frame from its Gram matrix.
This is possible both in the standard and symplectic settings, the latter of which we will now show.

In proving the result, we employ the spectral theorem for real skew symmetric matrices (see 2.5.11(b) in~\cite{horn2012matrix}) that states the following: 
a nonzero matrix \newline$A\in \R^{n\times n}$ with rank $2r$ is skew symmetric if and only if there exists a real orthogonal matrix $W$ such that 
\[
WAW^\top = 0_{n-2r} \oplus \mathcal{D}\Omega. 
\]
Here $\mathcal{D} = \operatorname{diag}(|\lambda_1|,|\lambda_1|, \ldots, |\lambda_r|,|\lambda_r|)$ where $\sigma(A) = \{0^{n-2r}, \pm i\lambda_1,\dots,\pm i\lambda_r\}$, counting multiplicities. 

\begin{theorem}
\label{thm:factoring}
    A matrix $G\in \R^{n\times n}$ is the Gram matrix of a frame in $\R_\mathcal{S}^d$ if and only if $G$ is skew symmetric with $\operatorname{rank}(G)=d$. Specifically, if $G$ is skew symmetric with $\sigma(G) = \{0^{n-d}, \pm i\lambda_1,\dots,\pm i\lambda_{d/2}\}$ then $G = \Phi^\dagger\Phi$ where $\Phi = DU$, $U U^\top = I_d$, and $D = \operatorname{diag}(\sqrt{|\lambda_1|},\sqrt{|\lambda_1|}, \ldots, \sqrt{|\lambda_{d/2}|},\sqrt{|\lambda_{d/2}|})$.
\end{theorem}

\begin{proof}
    The forward direction is clear. 
    In the reverse direction, let $G\in \R^{n\times n}$ be a skew symmetric matrix with $\operatorname{rank}(G) = d$ and $\sigma(G) = \{0^{n-d},\pm i\lambda_1,\dots,\pm i\lambda_{d/2}\}$.
    Let $D$ be as in the theorem statement. 
    By the spectral theorem for real skew symmetric matrices, there exists a real orthogonal matrix $W = \begin{bmatrix} V \\ U\end{bmatrix}$ in $\R^{n\times n}$, where $V\in \R^{(n-d)\times n}$ and $U\in \R^{d\times n}$ so that 
    \[ G = \begin{bmatrix}V^\top & U^\top \end{bmatrix}\begin{bmatrix} 0 & 0 \\ 0 & D^2\Omega\end{bmatrix}\begin{bmatrix}V \\ U \end{bmatrix} = (DU)^\dagger DU.\]
    We have $U U^\top = I_d$, hence $DU$ is surjective. 
    It follows from Proposition \ref{prop:spanning} that $\Phi = DU$ is a $d \times n$ frame in $\R_\mathcal{S}^d$. 
\end{proof}

Theorem \ref{thm:factoring} becomes essential for much of the rest of the paper, as it is the primary way we obtain frames in the symplectic setting. 
Moreover, the important properties we mentioned earlier are the concepts of tightness and \emph{equiangularity} (see Definition \ref{def:equiangular}).
We will see that these properties can be completely determined by the Gram matrix of a frame, so moving forward we begin to shift the object of study to Gram matrices.

As an important aside, while the above extracts a specific frame from a skew symmetric matrix, the underlying frame of the Gram matrix is not unique.
If $\Phi$ is a frame in $\R_\mathcal{S}^d$ that has Gram matrix $\Phi^\dagger\Phi$, then so does $M\Phi$ for any symplectic transformation $M$.
Therefore, in studying a frame through its Gram matrix, we are studying an equivalence class of frames determined by its Gram matrix.
This idea also offers insight into the degree to which a frame in $\R_\mathcal{S}^d$ and one in $\R_\mathcal{E}^d$ are equivalent.

\begin{theorem}
\label{thm:sympreal}
Let $\Phi$ be a sequence of vectors in $\R_\mathcal{S}^d$. 
Then $\Phi$ is a frame in $\R_\mathcal{S}^d$ with frame bounds $A,B$ if and only if there exists a symplectic transformation $M$ and a frame $\Psi$ in $\R_\mathcal{E}^d$ with frame bounds $A,B$ so that $\Phi = M\Psi$.
\end{theorem}

\begin{proof}
    For the forward direction, apply Theorem \ref{thm:factoring} to $\Phi^\dagger\Phi$. 
    The resulting frame in $\R_\mathcal{S}^d$, given by $\Psi = DU$, is also a frame in $\R_\mathcal{E}^d$.
    Denote $\Phi = \{\varphi_j\}_{j=1}^n$ and $\Psi = \{\psi_j\}_{j=1}^n$ and define the map $M:\R_\mathcal{S}^d\to\R_\mathcal{S}^d$ by $M(\psi_j) = \varphi_j$, expanding linearly. 
    To see this is well-defined, take $\sum_{j=1}^n a_j\psi_j = \sum_{j=1}^n b_j\psi_j$ and observe
    
    \begin{center}
    \resizebox{\textwidth}{!}{$
    M\left(\sum a_j\psi_j\right)-M\left(\sum b_j\psi_j\right) = 0 \iff \sum (a_j-b_j)\varphi_j = 0 \iff \begin{bmatrix}a_1-b_1 \\
    \cdots \\
    a_n-b_n
    \end{bmatrix}
    \in \ker \Phi.
    $}
    \end{center}
    
    Since $\Phi^\dagger$ and $\Psi^\dagger$ are injective, we have $\ker \Phi = \ker \Phi^\dagger\Phi = \ker \Psi^\dagger\Psi = \ker \Psi$.
    This proves $M$ is well-defined and clearly $\Phi = M\Psi$. 
    Moreover, $M$ is symplectic, as for any $x = \sum_{j=1}^na_j\psi_j$ and $y=\sum_{j=1}^nb_j\psi_j$ in $\R_\mathcal{S}^d$ we have 
    \[
    [Mx,My] = \sum_{i,j=1}^na_ib_j[\varphi_i,\varphi_j] = \sum_{i,j=1}^na_ib_j[\psi_i,\psi_j] = [x,y], 
    \]
    where the middle equality follows from the fact that $\Phi$ and $\Psi$ have the same Gram matrix over $\R_\mathcal{S}^d$.
    Finally, recall that $UU^\top = I_d$ hence
    \[
    \Psi\Psi^\top = DUU^\top D = D^2.
    \]
    The eigenvalues of $\Psi\Psi^\top$ are just its positive diagonal entries, which are exactly the absolute values of the eigenvalues of $\Phi\Phi^\dagger$.
    Thus $\Psi$ and $\Phi$ share frame bounds.

    For the reverse direction, let $M$ be a symplectic transformation and $\Psi$ be a $d\times n$ frame over $\R_{\mathcal{E}}^d$ with frame bounds $A$ and $B$ such that $\Phi = M\Psi$.
    It suffices to verify that $\Phi$ has frame bounds $A$ and $B$.

    Given any matrix $W$, denote by $|\lambda|_{\min}(W)$ and $|\lambda|_{\max}(W)$ the minimum and maximum absolute value of the eigenvalues of $W$, and denote by $\sigma_{\min}(W)$ and $\sigma_{\max}(W)$ the minimum and maximum singular values of $W$.
    First, note that $\sigma(\Phi\Phi^\dagger) = \sigma(\Psi\Psi^\dagger)$ since $\Phi^\dagger\Phi = \Psi^\dagger\Psi$.
    Because $\Psi\Psi^\top$ is positive definite, its eigenvalues and singular values coincide. 
    Furthermore, the singular values of $\Psi\Psi^\top$ and $\Psi\Psi^\dagger=\Psi\Psi^\top\Omega$ are the same, since $\Omega^\top\Omega = I_d$.
    Observe
    \[
    A\leq |\lambda|_{\min}(\Psi\Psi^\top) =\sigma_{\min}(\Psi\Psi^\top)= \sigma_{\min}(\Psi\Psi^\dagger)\leq |\lambda|_{\min}(\Psi\Psi^\dagger) = |\lambda|_{\min}(\Phi\Phi^\dagger),
    \]
    while also
    \[
    |\lambda|_{\max}(\Phi\Phi^\dagger) = |\lambda|_{\max}(\Psi\Psi^\dagger)\leq\sigma_{\max}(\Psi\Psi^\dagger) =\sigma_{\max}(\Psi\Psi^\top) =|\lambda|_{\max}(\Psi\Psi^\top)\leq B.
    \] 
    Therefore, we conclude that $A$ and $B$ are frame bounds for $\Phi$ over $\R_\mathcal{S}^d$.
\end{proof}

\subsection{Properties of Tight Frames in Real Symplectic Space}\,

Recall that a $c$-tight frame is one whose frame bounds $A$ and $B$ can be chosen as $A = B = c$ for some $c>0$.

\begin{example}\label{ex:tight}
    The matrix 
    \[
     G = \frac{1}{\sqrt{5}}\left[\begin{array}{rrr} 
     0 & 0 & -2 \\
     0 & 0 & 1 \\
     2 & -1 & 0
    \end{array} \right]
    \]
    is the Gram matrix of a $1$-tight frame in $\R_{S}^2$. 
    To see this, notice that the $\operatorname{rank}(G)=2$ so that applying Theorem \ref{thm:factoring} yields a frame 
    \[ 
    \Phi = \frac{1}{\sqrt{5}}\left[\begin{array}{rrr} -2 & 1 & 0 \\
    0 & 0 & \sqrt{5}
    \end{array}\right] \quad \text{which has frame operator } \quad \Phi\Phi^\dagger = \left[\begin{array}{rr} 0 & 1\\ -1 & 0 \end{array}\right].
    \]
    Although we have given $\Phi$ explicitly, this is not needed to see that any frame with Gram matrix $G$ is $1$-tight. 
    Theorem \ref{thm:factoring} ensures that $G=\Phi^\dagger\Phi$ for some $2\times 3$ frame $\Phi$ over $\R_\mathcal{S}^{2}$, while we know that $\sigma(\Phi^\dagger\Phi) = \{0, \pm i\}$. 
    Therefore it must be that $\sigma(\Phi\Phi^\dagger) = \{\pm i\}$.
\end{example}

In general, it is enough to consider only the Gram matrix of a frame when determining tightness. 
If $G$ is the Gram matrix of a $c$-tight frame in $\R_\mathcal{S}^d$, then $\sigma(G)\subseteq\{0,\pm ic\}$, ignoring multiplicities.
Meanwhile, the diagonalizability of $G$ implies that its minimal polynomial splits over $\C$. 
In particular, $G$ satisfies $G^3=-c^2G$. 
This provides a simple algebraic condition on $G$ when checking if the underlying frame is tight.

The following establishes the observation above in addition to symplectic versions of a few properties of tight frames in the complex or real Euclidean setting.

\begin{proposition}
\label{prop:tight}
    Let $\Phi$ be a sequence of $n$ vectors in $\R_\mathcal{S}^d$ and let $c>0$. The following are equivalent:
    \begin{enumerate}
        \item[(i)] $\Phi$ is a $c$-tight frame.
        \item[(ii)] $\Phi = M\Psi$, for some symplectic transformation $M$ and real $c$-tight frame $\Psi$.
        \item[(iii)] $(\Phi\Phi^\dagger)^2 = -c^2I_d$. 
        \item[(iv)] $(\Phi^\dagger\Phi)^3 = -c^2\Phi^\dagger\Phi$ with $\operatorname{rank}(\Phi^\dagger\Phi)=d$. 
        \item[(v)] The nonzero singular values of $\Phi^\dagger\Phi$ are $\sigma_1=\cdots=\sigma_d = c$.
    \end{enumerate}
\end{proposition}

\begin{proof}
    The implication $(i) \Rightarrow (ii)$ is given by Theorem \ref{thm:sympreal}. 
    Meanwhile $(ii) \Rightarrow (iii)$ is straightforward after applying the identity $\Psi\Psi^\top = cI_d$.
    For $(iii) \Rightarrow (iv)$, a short calculation gives $(\Phi^\dagger\Phi)^3 = -c^2\Phi^\dagger\Phi$, while 
    \[
    d = \operatorname{rank}(\Phi\Phi^\dagger) \leq \operatorname{rank}(\Phi) \leq d
    \]
    implies that $\Phi$ is a frame by Proposition \ref{prop:spanning}.
    Hence $\ker \Phi = \ker \Phi^\dagger\Phi$ and $\operatorname{rank}(\Phi^\dagger\Phi) = d$.
    
    To see $(iv)\Rightarrow (v)$, we first note that $\Phi^\dagger\Phi$ has exactly $d$ nonzero singular values on account of its rank.
    Now, observe that the singular values of $\Phi^\dagger\Phi$ are the square roots of the eigenvalues of $-(\Phi^\dagger\Phi)^2$.
    Recall that $\Phi^\dagger\Phi$ is real and skew symmetric, so that it is diagonalizable and its eigenvalues come in purely imaginary conjugate pairs. 
    Moreover, $\ker\Phi = \ker\Phi^\dagger\Phi$ implies that the sum of the algebraic multiplicities of the nonzero eigenvalues of $\Phi^\dagger\Phi$ is $d$.
    Therefore $\sigma(\Phi^\dagger\Phi) = \{ic^{d/2}, -ic^{d/2}, 0^{n-d}\}$ (here the exponents represent multiplicities) and so the eigenvalues of $-(\Phi^\dagger\Phi)^2$ are $c^2$ with multiplicity $d$ and $0$ with multiplicity $n-d$.
    It follows that the $d$ nonzero singular values of $\Phi^\dagger\Phi$ are all equal to $c$.

    Finally, to see that $(v) \implies (i)$, we observe that if $\Phi^\dagger\Phi$ has exactly $d$ nonzero singular values equal to $c$, then for reasons provided in the paragraph above, we see that $\sigma(\Phi^\dagger\Phi) = \{ic^{d/2}, -ic^{d/2}, 0^{n-d}\}$.
    Since $\Phi^\dagger\Phi$ shares nonzero eigenvalues (including algebraic multiplicities) with the $d\times d$ symplectic frame operator $\Phi\Phi^\dagger$, it follows that $|\lambda|=c$ for all $\lambda\in \sigma(\Phi\Phi^\dagger)$.
    That is, $\Phi$ is a $c$-tight frame.
    \end{proof}

\begin{remark}
    A frame in the complex or real Euclidean setting is tight if and only if its frame operator is a positive multiple of the identity.
    Item $(ii)$ above shows that a frame in $\R_\mathcal{S}^d$ is tight if and only if there exists an equivalent frame whose frame operator is a positive multiple of $\Omega$. 
    Specifically, if $\Phi= M_0\Psi$ for some $c$-tight frame $\Psi$ in $\R_{\mathcal{E}}^d$, then $M_0^{-1}\Phi$ has frame operator $c\Omega$.
    The additional condition of equivalence is needed here: for instance, Example~\ref{ex:basic} shows a symplectic-tight frame (in fact, by definition all frames in $\R_\mathcal{S}^2$ must be tight) whose frame operator is not a multiple of $\Omega$.
\end{remark}

For the remainder of the paper, we use the condition given by $(iv)$ in Proposition \ref{prop:tight} to determine if a frame is tight. 
The usefulness of this characterization becomes apparent later on.

\section{Equiangular Tight Frames in Real Symplectic Space}
In this section, we discuss equiangular and equiangular tight frames in the symplectic setting.
At the end, we prove both of our main results.

In the real setting, an equiangular frame is one whose columns are scaled to be unit-norm and whose Gram matrix has off-diagonal entries with uniform modulus. 
This informs an analogous definition.

\begin{definition}\label{def:equiangular}
    A sequence of vectors $\Phi = \{\varphi_i\}_{i=1}^n$ in $\R_{\mathcal{S}}^d$ is \textbf{equiangular} if there exists an $\mu>0$ such that $|\Phi^\dagger\Phi_{ij}| = \mu$ whenever $i\neq j$.
\end{definition}

In the symplectic setting, we lack a norm, and in particular the diagonal entries of the Gram matrix are always zero. 
Therefore by the definition above, we may equally scale the vectors of an equiangular sequence to achieve another equiangular sequence, and this is something we will occasionally do in the proofs of results to follow. 

In general, equiangularity alone is not of interest to us.
Rather, we want to find \textbf{equiangular tight frames} (\textbf{ETF}) in $\R_\mathcal{S}^d$: that is, frames that are both tight and equiangular in the symplectic setting.
These objects, in particular, can be characterized as solutions to optimization problems. 
To that end, we define the $p$th order symplectic frame potential analogously to the frame potential found in the complex or real Euclidean setting (see~\cite{benedetto2003finite, renes2004symmetric, oktay2007frame}).

\begin{definition}
    The \textbf{$p$th order symplectic frame potential} of a sequence \newline $\Phi = \{\varphi_i\}_{i=1}^n$ in $\R_\mathcal{S}^d$ is defined as
    \[ 
    \mathcal{SF}_p(\Phi) = \begin{cases}\displaystyle\sum_{i,j=1}^n\big|[\varphi_i,\varphi_j]\big|^{2p} &\text{ for } 1\leq p <\infty,\\[15pt] \displaystyle\max_{i\neq j} \big|[\varphi_i,\varphi_j]\big| & \text{ for } p=\infty.\end{cases}
    \]
\end{definition}

Results in the complex and real Euclidean settings analogous to the following theorem involve restricting an ETF to have unit-norm vectors. 
Here, fixing the nuclear norm of $\Phi^\dagger\Phi$ is the analogous restriction.

\begin{theorem}\label{thm:optimETF}
    Let $\Phi = \{\varphi_i\}_{i=1}^n$ be a sequence of $n\geq d$ vectors in $\R_\mathcal{S}^d$, with $\|\Phi^\dagger\Phi\|_*= \sqrt{dn(n-1)}$ (where $\|\cdot\|_*$ denotes the nuclear norm). 
    \begin{enumerate}
        \item[(a)] The first order symplectic frame potential of $\Phi$ satisfies
        \[\mathcal{SF}_1(\Phi) \geq \sqrt{\frac{n(n-1)}{d}},\]
        and equality holds if and only if $\Phi$ is a tight frame.
        \item[(b)] Let $p\in (1,\infty]$. 
        Then
        \[
        \mathcal{SF}_p(\Phi)\geq \begin{cases} n(n-1)& \text{if }1< p < \infty, \\
        1 & \text{if } p = \infty.
        \end{cases}
        \]
        and equality holds if and only if $\Phi$ is an equiangular tight frame.
    \end{enumerate}
\end{theorem}

\begin{proof}
    Let $\sigma_1\geq\cdots\geq\sigma_n\geq 0$ be the singular values of $\Phi^\dagger\Phi$, of which at most $d$ are nonzero.
    We first prove $(a)$.
    An application of Cauchy-Schwartz yields
    \[
    \displaystyle\sum_{i,j=1}^n\big|[\varphi_i,\varphi_j]\big|^{2}= \sum_{i=1}^d \sigma_i^2 \geq \frac{1}{d}\left(\sum_{i=1}^d \sigma_i\right)^2 = \frac{\|\Phi^\dagger\Phi\|_*^2}{d} = \sqrt{\frac{n(n-1)}{d}}.
    \]
    Equality holds if and only $\Phi^\dagger\Phi$ has exactly $d$ nonzero singular values, all of which are equal.
    By Proposition \ref{prop:tight}, this is equivalent to $\Phi$ being a tight frame.
    
    Let $p\in (1,\infty)$ and let $q$ be the H\"{o}lder conjugate of $p$, that is $\frac{1}{p}+\frac{1}{q}=1$. 
    Observe
    \begin{align*}
        \displaystyle\sum_{i,j=1}^n\big|[\varphi_i,\varphi_j]\big|^{2p} =\sum_{i\neq j}\big|[\varphi_i,\varphi_j]\big|^{2p} 
        &\geq\left(\frac{1}{n(n-1)}\right)^{p/q}\left(\sum_{i\neq j}\big|[\varphi_i,\varphi_j]\big|^2\right)^p \\
        &=\left(\frac{1}{n(n-1)}\right)^{p/q}\left(\sum_{i=1}^d \sigma_i^2\right)^p \\
        &\geq \left(\frac{1}{n(n-1)}\right)^{p/q}\left(\frac{1}{d}\left(\sum_{i=1}^d \sigma_i\right)^2\right)^p \\
        &= \left(\frac{1}{d^qn(n-1)}\right)^{p/q}\|\Phi^\dagger\Phi\|_*^{2p} \\
        &= \left(\frac{\big(dn(n-1)\big)^q}{d^qn(n-1)}\right)^{p/q} \\
        &= n(n-1).
    \end{align*}
    The first inequality is H\"{o}lder's inequality, so equality holds if and only if $\big|[\varphi_i,\varphi_j]\big|^2$ is constant for all $i\neq j$, i.e.\ when $\Phi$ is equiangular.
    Equality holds in the second inequality if and only if $\Phi^\dagger\Phi$ has exactly $d$ nonzero singular values, all of which are equal.
    By Proposition \ref{prop:tight}, this occurs if and only if $\Phi$ is a tight frame.
    
    When $p = \infty$, we have
    \[
    \max_{i\neq j}|[\varphi_i,\varphi_j]|^2 \geq \frac{1}{n(n-1)}\sum_{i\neq j}|[\varphi_i,\varphi_j]|^2 
    \geq \frac{1}{dn(n-1)}\left(\sum_{i=1}^d\sigma_i\right)^2 
    =  \frac{\|\Phi^\dagger\Phi\|_*^{2}}{dn(n-1)} = 1.
    \]
    Once again, equality occurs throughout if and only if $|[\varphi_i,\varphi_j]|^2$ is constant for all $i\neq j$ and $\Phi^\dagger\Phi$ has exactly $d$ nonzero singular values, all of which are equal. 
    Thus as above, equality holds throughout if and only if $\Phi$ is an ETF. 
\end{proof}

\begin{remark}
Theorem \ref{thm:optimETF} establishes that ETFs in real symplectic space, when they exist, are solutions to an optimization problem.
Conceivably, one could leverage Theorem \ref{thm:equiv} to find skew Hadamard matrices by using an optimization routine to search for ETFs in real symplectic space. 
Such a task is beyond the scope of this paper and so we leave this as an open problem.
\end{remark} 

We now give some basic examples of ETFs in $\R_\mathcal{S}^d$.

\begin{example}\label{ex:basicETF}
    The $2\times 2$ frame $\Phi = I_2$ is a $1$-tight ETF in $\R_\mathcal{S}^2$, since 
    \[\Phi^\dagger\Phi = \Omega = \left[\begin{array}{rr} 0 & 1\\ -1 & 0\end{array}\right] \qquad \text{ and } \qquad (\Phi^\dagger\Phi)^3 = \Omega^3 = -\Omega = -\Phi^\dagger\Phi.\]
\end{example}

\begin{example}\label{ex:conf}
    The matrix 
    \[
    C = \left[\begin{array}{rrrr} 0 & 1 & 1 & 1 \\
    -1 & 0 & -1 & 1 \\
    -1 & 1 & 0 & -1 \\
    -1 & -1 & 1 & 0 \end{array}\right]
    \]
    is the Gram matrix of a $4\times 4$ ETF in real symplectic space.
    We have $\operatorname{rank}(C) = 4$ and so applying Theorem $\ref{thm:factoring}$ gives a $4\times 4$ frame in $\R_{\mathcal{S}}^4$. 
    Meanwhile the underlying frame is clearly equiangular while one can check that $C^3 = -3C$, that is, the frame is $\sqrt{3}$-tight.

    The principal submatrix 
    \[
    K = \left[\begin{array}{rrr} 0 & -1 & 1 \\
    1 & 0 & -1 \\
    -1 & 1 & 0 \end{array}\right]
    \]
    of $C$ is the Gram matrix of a $2\times 3$ ETF in real symplectic space, as we can similarly check that $\operatorname{rank}(K) = 2$ and $K^3 = -3K$.
\end{example}

The matrix $C$ in the previous example is a \emph{skew conference matrix}, as is the Gram matrix in Example \ref{ex:basicETF}. 
A \textbf{skew conference matrix} $C$ of order $n$ is an $n\times n$ matrix with off-diagonal entries $\pm 1$ satisfying $C = -C^\top$ and $CC^\top = (n-1)I_n$.
For any skew symmetric matrix $A$ one has 
\[
AA^\top = (n-1)I_n \iff (A+I)(A+I)^\top = nI_n.
\]
In particular, this means that $C$ is a skew conference matrix if and only if $C+I$ is a skew Hadamard matrix.
Thus $n\times n$ skew Hadamard matrices exist if and only if $n\times n$ skew conference matrices exist.  

One can normalize any skew conference matrix $C$ of order at least $2$ to produce a unique skew conference matrix $C'$ with the form: 
\[
C' = \left[\begin{array}{rr} 0 & \mathbf{1}^\top \\ -\mathbf{1} & K \end{array}\right],
\]
where $\mathbf{1}$ denotes the all-ones vector whose dimension is clear from context.
This normalization is done by conjugating $C$ with a $\pm 1$-diagonal matrix.
We call the $(n-1)\times (n-1)$ matrix $K$ the \textbf{core} of $C$. 
In Example \ref{ex:conf} above, $C$ is a skew conference matrix that has already been normalized, and $K$ is its core.
Actually, skew conference matrices of order at least 2 always give rise to ETFs in real symplectic space.
We explain in the following theorem, which proves three of the four implications in Theorem \ref{thm:equiv}. 

\begin{theorem}\label{thm:confgivesetf}
    Choose any $d>1$.
    \begin{enumerate}
        \item[(a)] There exists a $d\times d$ ETF in real symplectic space if and only if there exists a skew Hadamard matrix of order $d$. 
        Specifically, $G\in \R^{d\times d}$ is the Gram matrix of such an ETF if and only if $I_d+\frac{1}{|G_{12}|}G$ is a skew Hadamard matrix.
        \item[(b)] There exists a $d\times (d+1)$ ETF in real symplectic space if there exists a skew Hadamard matrix of order $d+2$. Specifically, if $H\in \R^{(d+2)\times (d+2)}$ is a skew Hadamard matrix, then $H-I_{d+2}$ is a skew conference matrix, and its core is the Gram matrix of a $d\times (d+1)$ ETF in real symplectic space.
    \end{enumerate}
\end{theorem}

\begin{proof}
    First we prove the forward direction of $(a)$. 
    Let $G$ be the Gram matrix of a $d\times d$ ETF in real symplectic space. 
    It follows that $|G_{12}| = |G_{ij}|$ for all $i\neq j$, and so $H:=I_d + \frac{1}{|G_{12}|}G$ has entries $\pm 1$. 
    It remains to verify that $HH^\top = dI_d$.
    We know that $G^3=-c^2G$ for some $c>0$, and that $G$ is invertible since it is the Gram matrix of a $d\times d$ ETF.
    This gives that $GG^\top = -c^2I_d$. 
    Since the diagonal entries of $GG^\top$ are the Euclidean norm squared of the rows of $G$, we see that $-c^2 = |G_{12}|^2(d-1)$ and so
    \[
    HH^\top = I_d + \frac{1}{|G_{12}|^2}GG^{\top} = I_d+(d-1)I_d = dI_d.
    \]
    
    For the reverse direction, let $H=I_d+\frac{1}{|G_{12}|}G$ be a skew Hadamard matrix.
    Thus $C:=\frac{1}{|G_{12}|}G$ is a skew conference matrix, which has off-diagonal entries $\pm 1$, is skew symmetric, and satisfies $CC^\top = (d-1)I_d$.
    It follows that $C^3 = -(d-1)C$ and $\operatorname{rank}(C) = d$.
    Since the off-diagonal entries of $C$ are $\pm 1$, applying Theorem \ref{thm:factoring} yields that $C$, and therefore $G$, is the Gram matrix of a $d\times d$ ETF in real symplectic space. 

    For $(b)$, let $H\in \R^{(d+2)\times(d+2)}$ be a skew Hadamard matrix so that $C =H - I_{d+2}$ is a skew conference matrix, where without loss of generality $C$ is normalized.
    Let $K$ be the core of $C$.
    From $CC^\top = (d-1)I_d$, one sees that $\mathbf{1}\in \ker K$ and $KK^\top = (d-1)I_{d-1}-J_{d-1}$, where $J_{d-1}$ is the $(d-1)\times (d-1)$ all ones matrix.
    Thus $K^3 = -(d-1)K$. 
    Examining the spectrum of $KK^\top$ indicates that $K$ has $d-2$ nonzero singular values, hence $\operatorname{rank}(K) = d-2$. 
    Meanwhile, the off-diagonal entries of $K$ are $\pm 1$.
    An application of Theorem \ref{thm:factoring} yields a $(d-2)\times (d-1)$ ETF for which $K$ is the Gram matrix.
\end{proof}

\subsection{Symplectic Zauner's Conjecture}\,

Recall that in the complex or real Euclidean setting, Gerzon's bound provides an upper bound for the number of vectors in an ETF. 
So far in the symplectic setting, we have only seen examples of $d\times d$ ETFs (skew conference matrices) and $d\times (d+1)$ ETFs (their cores).
It is natural to ask whether $d\times n$ ETFs in real symplectic space exist for $n>d+1$. 
To that end, we provide a symplectic version of Gerzon's bound.
The naming of this result is intentional and appropriate: the original Gerzon's bound is proved by an argument relying on linear independence, and the proof below relies on linear independence as well.

\begin{proposition}[Symplectic Gerzon's Bound]\label{prop:gerzon} 
    Let $\Phi$ be an equiangular sequence of $n$ vectors in $\R_\mathcal{S}^d$. Then $n\leq d+1$.
\end{proposition}

\begin{proof}
    Without loss, we may assume that $|\Phi^\dagger\Phi_{ij}|=1$ for all $i\neq j$.
    
    First, suppose $n$ is even.
    The $n\times n$ Gram matrix $\Phi^\dagger\Phi$ has zeros along the main diagonal and $\pm 1$ elsewhere. 
    Denote by $\pi(\Phi^\dagger\Phi)$ the application of natural projection $\pi: \Z \to \Z/2\Z$ to the entries of $\Phi^\dagger\Phi$. 
    This projection is a ring homomorphism and $\det(\Phi^\dagger\Phi)$ is a polynomial in the entries of $\Phi^\dagger\Phi$, so $\det(\pi(\Phi^\dagger\Phi)) = \pi(\det(\Phi^\dagger\Phi))$. 
    Note also that $\pi(\Phi^\dagger\Phi) = \pi(J_n-I_n)$. Since $\det(J_n-I_n) = (-1)^{n-1}(n-1)$, and since $n$ is even,
    \[ 
    1 = \pi(\det(J_n-I_n))=\det(\pi(J_n-I_n)) = \det(\pi(\Phi^\dagger\Phi)) = \pi(\det(\Phi^\dagger\Phi)).
    \] 
    Therefore $\Phi^\dagger\Phi$ has odd, necessarily nonzero, determinant. Thus the columns of $\Phi$ are linearly independent, as $\ker \Phi\subseteq\ker\Phi^\dagger\Phi = \{0\}$. So $n\leq d$.

    If $n$ is odd, remove one vector from $\Phi$ and apply the even case to see $n-1 \leq d$.
\end{proof}

The following result identifies the frame bounds of an ETF in $\R_\mathcal{S}^d$. 

\begin{proposition}\label{prop:cequals}
    Let $\Phi$ be a $c$-tight $d\times n$ ETF in $\R_\mathcal{S}^d$ with $|\Phi^\dagger\Phi_{ij}| = \mu$ for $i \neq j$. Then 
    \[
    c = \begin{cases}
        \mu\sqrt{n-1} & \text{if }n=d, \\
        \mu\sqrt{n} & \text{if }n = d +1.
    \end{cases}
    \]
\end{proposition}

\begin{proof}
    Note that Proposition \ref{prop:gerzon} gives $n = d$ or $n=d+1$. 
    By Proposition \ref{prop:tight}, $(\Phi\Phi^\dagger) = -c^2I_d$. 
    Thus  
    \[
    c^2d = -\operatorname{tr}\left((\Phi\Phi^\dagger)^2\right) = -\operatorname{tr}\left((\Phi^\dagger\Phi)^2\right) =  \operatorname{tr}\big((\Phi^\dagger\Phi)(\Phi^\dagger\Phi)^\top\big) = \|\Phi^\dagger\Phi\|_F^2= \mu^2(n-1)n.
    \]
    Dividing by $d>0$ gives the desired result.
\end{proof}

To this point, we have proved three of the four implications in Theorem \ref{thm:equiv}. 
What remains is the forward direction of $(b)$.
As one might expect, this direction is not as obvious.

In the proof of this remaining implication, we will show that the Gram matrix of a (scaled) $d\times (d+1)$ ETF in $\R_{\mathcal{S}}^d$ is the core of a skew conference matrix. 
To do this, we will need to show that the kernel of said Gram matrix is spanned by a flat vector, that is, a vector whose entries are of equal modulus.
This fact is not immediately clear and requires one to uncover and study the combinatorial object that underlies ETFs in real symplectic space. 
These objects belong to a certain class of directed graphs.

Consider the Gram matrix $K$ in Example \ref{ex:conf}.
This matrix determines a directed graph.
Specifically, one can assign 3 vertices $v_1,v_2$, and $v_3$ to the rows and columns of $K$.
Then whenever $[K]_{ij}=1$, draw an edge from $v_i$ to $v_j$.
The matrix $K$ and the resulting graph are shown in Figure \ref{fig:graphandgram}.

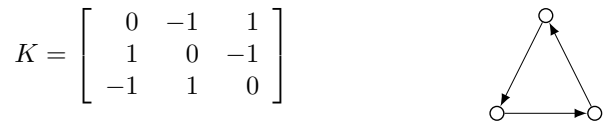
\begin{figure}[h!]
    \begin{multicols}{2}
    \qquad\qquad\qquad
    $K = \left[\begin{array}{rrr} 0 & -1 & 1 \\
    1 & 0 & -1 \\
    -1 & 1 & 0 \end{array}\right] $ \columnbreak
    
    \scalebox{1.3}{
    \begin{tikzpicture}
        \node[shape=circle,draw=black,inner sep=0pt,minimum size=4pt] (2) at (0,1/2) {};
        \node[shape=circle,draw=black,inner sep=0pt,minimum size=4pt] (3) at (-1/2,-1/2) {};
        \node[shape=circle,draw=black,inner sep=0pt,minimum size=4pt] (4) at (1/2,-1/2) {};

       %\path [<-, line width = 0.1mm, >=latex] (1) edge node[left] {} (2);
       %\path [<-, line width = 0.1mm, >=latex] (1) edge node[left] {} (3);
       %\path [<-, line width = 0.1mm, >=latex] (1) edge node[left] {} (4);
        \path [->, line width = 0.1mm, >=latex] (2) edge node[left] {} (3);
        \path [<-, line width = 0.1mm, >=latex] (2) edge node[left] {} (4);
        \path [->, line width = 0.1mm, >=latex] (3) edge node[left] {} (4);
    \end{tikzpicture}}\qquad\qquad
    \end{multicols}
    \caption{$K$ and its corresponding graph}
    \label{fig:graphandgram}
\end{figure}

The upshot of this relationship is that by leveraging results based on the structure of its corresponding graph, we can, given a $d\times (d+1)$ ETF in $\R_\mathcal{S}^d$, determine whether or not the kernel of its Gram matrix is spanned by a flat vector. 
We detail these results in the next subsection.

\subsection{Background on Tournaments and Diamonds}\,

The following subsection provides sufficient background on counting subgraphs in tournaments as found in~\cite{belkouche2020matricial}.
A more detailed explanation, including proofs of the results listed here, can be found there.

\begin{definition}
    An \textbf{n-tournament} $T$ is a directed graph on $n$ vertices with one edge between every pair of vertices. 
\end{definition}

In a tournament, a vertex is $i$ said to \textbf{dominate} (resp.\ \textbf{be dominated by}) a vertex $j$ if there exists an edge from $i$ to $j$ (resp.\ $j$ to $i$). 
The set of vertices dominated by (resp.\ dominating) vertex $i$ is denoted $N_T^+(i)$ (resp.\ $N_T^-(i)$). 
Also, set
\[ 
d_T^+(i) = |N_T^+(i)|, \qquad d_T^-(i) = |N_T^-(i)|,
\]
\[
d_T^+(i,j) = |N_T^+(i)\cap N_T^+(j)|, \qquad d_T^-(i,j) = |N_T^-(i)\cap N_T^-(j)|.
\]

A tournament is completely determined by its \emph{Seidel adjacency matrix}.
The \textbf{Seidel adjacency matrix} $S\in\R^{X\times X}$ of a tournament $T$ on vertex set $X$ is defined by
\[
S_{ij} = \begin{cases}
    1 & \text{if $i$ dominates $j$}, \\
    -1 & \text{if $i$ is dominated by $j$}, \\
    0 & \text{otherwise}.
\end{cases}
\]

An important fact regarding Seidel adjacency matrices is that their spectrum remains invariant under \textbf{switching}, or conjugating by a diagonal $\pm1$ matrix.
Two tournaments $T$ and $T'$ on the same vertex set are \textbf{switching equivalent} if there exists a diagonal $\pm1$ matrix $D$ such that $DSD = S'$, where $S$ and $S'$ are the corresponding Seidel adjacency matrices of $T$ and $T'$.

Up to isomorphism, there are four $4$-tournaments: two whose Seidel adjacency matrices have determinant $1$ and two that have determinant $9$.
The $4$-tournaments with determinant $9$ are called \textbf{diamonds} and are given by a single vertex dominating, or being dominated by, a $3$-cycle.
The number of diamonds in a tournament $T$ is denoted by $\delta_T$.
Figure \ref{fig:fourfours} shows the four different $4$-tournaments up to isomorphism, with the diamonds being the two tournaments on the left.

\begin{figure}[h!]
\begin{center}
    \scalebox{1.2}{
    \begin{tikzpicture}
        \node[shape=circle,draw=black,inner sep=0pt,minimum size=4pt] (1) at (0,-.16) {};
        \node[shape=circle,draw=black,inner sep=0pt,minimum size=4pt] (2) at (0,1/2) {};
        \node[shape=circle,draw=black,inner sep=0pt,minimum size=4pt] (3) at (-1/2,-1/2) {};
        \node[shape=circle,draw=black,inner sep=0pt,minimum size=4pt] (4) at (1/2,-1/2) {};

       \path [<-, line width = 0.1mm, >=latex] (1) edge node[left] {} (2);
       \path [<-, line width = 0.1mm, >=latex] (1) edge node[left] {} (3);
       \path [<-, line width = 0.1mm, >=latex] (1) edge node[left] {} (4);
        \path [->, line width = 0.1mm, >=latex] (2) edge node[left] {} (3);
        \path [<-, line width = 0.1mm, >=latex] (2) edge node[left] {} (4);
        \path [->, line width = 0.1mm, >=latex] (3) edge node[left] {} (4);
    \end{tikzpicture}}\qquad
    \scalebox{1.2}{\begin{tikzpicture}
        \node[shape=circle,draw=black,inner sep=0pt,minimum size=4pt] (1) at (0,-.16) {};
        \node[shape=circle,draw=black,inner sep=0pt,minimum size=4pt] (2) at (0,1/2) {};
        \node[shape=circle,draw=black,inner sep=0pt,minimum size=4pt] (3) at (-1/2,-1/2) {};
        \node[shape=circle,draw=black,inner sep=0pt,minimum size=4pt] (4) at (1/2,-1/2) {};

        \path [->, line width = 0.1mm, >=latex] (1) edge node[left] {} (2);
        \path [->, line width = 0.1mm, >=latex] (1) edge node[left] {} (3);
        \path [->, line width = 0.1mm, >=latex] (1) edge node[left] {} (4);
        \path [->, line width = 0.1mm, >=latex] (2) edge node[left] {} (3);
        \path [<-, line width = 0.1mm, >=latex] (2) edge node[left] {} (4);
        \path [->, line width = 0.1mm, >=latex] (3) edge node[left] {} (4);
    \end{tikzpicture}}\qquad
    \scalebox{1.2}{\begin{tikzpicture}
        \node[shape=circle,draw=black,inner sep=0pt,minimum size=4pt] (1) at (0,-.16) {};
        \node[shape=circle,draw=black,inner sep=0pt,minimum size=4pt] (2) at (0,1/2) {};
        \node[shape=circle,draw=black,inner sep=0pt,minimum size=4pt] (3) at (-1/2,-1/2) {};
        \node[shape=circle,draw=black,inner sep=0pt,minimum size=4pt] (4) at (1/2,-1/2) {};

        \path [<-, line width = 0.1mm, >=latex] (1) edge node[left] {} (2);
        \path [<-, line width = 0.1mm, >=latex] (1) edge node[left] {} (3);
        \path [<-, line width = 0.1mm, >=latex] (1) edge node[left] {} (4);
        \path [->, line width = 0.1mm, >=latex] (2) edge node[left] {} (3);
        \path [->, line width = 0.1mm, >=latex] (2) edge node[left] {} (4);
        \path [->, line width = 0.1mm, >=latex] (3) edge node[left] {} (4);
    \end{tikzpicture}}
    \qquad\scalebox{1.2}{\begin{tikzpicture}
        \node[shape=circle,draw=black,inner sep=0pt,minimum size=4pt] (1) at (0,-.16) {};
        \node[shape=circle,draw=black,inner sep=0pt,minimum size=4pt] (2) at (0,1/2) {};
        \node[shape=circle,draw=black,inner sep=0pt,minimum size=4pt] (3) at (-1/2,-1/2) {};
        \node[shape=circle,draw=black,inner sep=0pt,minimum size=4pt] (4) at (1/2,-1/2) {};

        \path [<-, line width = 0.1mm, >=latex] (3) edge node[left] {} (2);
        \path [->, line width = 0.1mm, >=latex] (3) edge node[left] {} (1);
        \path [->, line width = 0.1mm, >=latex] (3) edge node[left] {} (4);
        \path [<-, line width = 0.1mm, >=latex] (2) edge node[left] {} (1);
        \path [<-, line width = 0.1mm, >=latex] (2) edge node[left] {} (4);
        \path [->, line width = 0.1mm, >=latex] (1) edge node[left] {} (4);
    \end{tikzpicture}}
    \end{center}
    \caption{The four $4$-tournaments up to isomorphism}
    \label{fig:fourfours}
\end{figure}
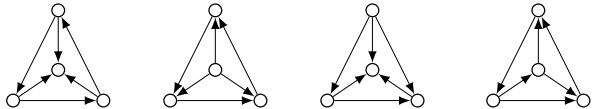

In~\cite{belkouche2020matricial}, Belkouche et al.\ provide upper bounds on $\delta_T$ and in some cases, conditions for saturation.
In particular, when the number of vertices is odd we have the following.

\begin{proposition}[Proposition 9 in~\cite{belkouche2020matricial}]
\label{prop:belkoucheprop}
Let $T$ be an $n$-tournament and $S$ its corresponding Seidel adjacency matrix. 
If $n$ is odd, then
\[
\delta_T \leq \frac{1}{96}n(n-1)(n-3)(n+1).
\]
Moreover, equality holds if and only if $|S^2+nI_n|=J_n$, where $|\cdot|$ is the entry-wise absolute value. 
\end{proposition}

A \textbf{doubly regular tournament} is an $n$-tournament where $d_T^+(i) = d_T^-(i)=\frac{n-1}{2}$ for all vertices $i$ and $d_T^+(i,j) = \frac{n-3}{4}$ for all vertices $i\neq j$. 
By integrality constraints, a doubly regular tournament must have $n \equiv 3$ mod $4$ vertices.

The next theorem, which is essential to the proof of our second main result, shows that when $n\equiv 3$ mod 4, equality in Proposition \ref{prop:belkoucheprop} is the same as $T$ being switching equivalent to a doubly regular tournament.

\begin{proposition}[Theorem 10 in~\cite{belkouche2020matricial}]
\label{thm:belkouchethm}
     Let $T$ be an $n$-tournament and $S$ its Seidel adjacency matrix. 
     If $n\equiv 3$ mod $4$, the following assertions are equivalent:
    \begin{enumerate}
        \item[(i)] $\delta_T = \frac{1}{96}n(n-1)(n-3)(n+1)$.
        \item[(ii)] There exists a $\pm 1$-diagonal matrix such that $DS^2D + nI_n = J_n$.
        \item[(iii)] $T$ is switching equivalent to a doubly regular tournament.
    \end{enumerate}
\end{proposition}

It is not hard to see that an equiangular sequence of $n$ vectors in $\R_\mathcal{S}^d$ can be scaled so that the corresponding Gram matrix is the Seidel adjacency matrix of an $n$-tournament.
Even further, we will soon see that the Gram matrix of a $d\times (d+1)$ ETF in $\R_\mathcal{S}^d$ is the Seidel adjacency matrix of a tournament that is switching equivalent to a doubly regular tournament. 
The significance of this is the following: the Seidel adjacency matrix of a doubly regular tournament has a kernel spanned by the all ones vector.
Therefore, any tournament that is switching equivalent to a doubly regular tournament has a Seidel adjacency matrix whose kernel is spanned by a flat vector.

Lastly, the following computational result assists us in proving Lemma \ref{lem:not1mod4}.

\begin{lemma}[From the proof of Proposition 6 in~\cite{belkouche2020matricial}]
\label{lem:belkouchelem}
Let $T$ be an $n$-tournament with $n>1$ and $S$ its Seidel adjacency matrix.
Define 
\[
\gamma_{ij} := |N_T^+(i)\cap N_T^-(j)|+|N_T^-(i)\cap N_T^+(j)|.
\] 
Then the following identities hold:
\[
\gamma_{ij} = 2n-3 - (d_T^+(i)+d_T^-(i)+2d_T^+(i,j)), \qquad \qquad (S^2)_{ij} = 2\gamma_{ij}-n+2.
\]
\end{lemma}

\subsection{Proof of Main Results}

Before we prove the main results, we give the following lemma.
The proof uses the same technique as the one used for part of Proposition \ref{thm:belkouchethm} (Theorem 10 in~\cite{belkouche2020matricial}).

\begin{lemma}
\label{lem:not1mod4}
    Let $T$ be an $n$-tournament with $n>1$ and $S$ its Seidel adjacency matrix.
    If $n\equiv 1$ mod $4$, then
    \[\delta_T < \frac{1}{96}n(n-1)(n-3)(n+1).\]
\end{lemma}

\begin{proof}
    Suppose for sake of contradiction that equality holds.
    By Proposition \ref{prop:belkoucheprop}, $(S^2)_{ij} = \pm 1$ for all $i\neq j$.
    First, suppose $d_T^+(i) \equiv d_T^+(j)$ mod $2$. 
    Then Lemma \ref{lem:belkouchelem} and a few straightforward calculations imply that $\gamma_{ij}$ is odd, and subsequently $(S^2)_{ij} \equiv 3$ mod $4$. 
    Thus $(S^2)_{ij} = -1$. 
    If instead $d_T^+(i) \equiv d_T^+(j)+1$ mod $2$, we see that $\gamma_{ij}$ is even and $(S^2)_{ij} \equiv 1$ mod $4$, hence $(S^2)_{ij} = 1$.
    
    Construct $D=\text{diag}(\varepsilon_1,\dots,\varepsilon_n)$ where $\varepsilon_i = 1$ if $d_T^+(i)$ is even and $-1$ otherwise. 
    Since $[DS^2D]_{ij} = \varepsilon_i(S^2)_{ij}\varepsilon_j$, we have
    \[ 
    DS^2D = \begin{bmatrix}
        1-n & & -1 \\
        & \ddots & \\
        -1 & & 1-n
    \end{bmatrix} = (2-n)I_n - J_n. 
    \]
    This implies that $\sigma(S^2)=\sigma(DS^2D) = \{2-2n, 2-n\}$. 
    In particular, $S$ has full rank, since $S^2$ does. 
    But $S$ is an odd-sized, real skew symmetric matrix and so must be rank-deficient, a contradiction.
\end{proof}
Now we are ready to prove both of the main results. 
Part of the following proofs uses a technique from~\cite{belkouche2020matricial} involving the characteristic polynomial.

\begin{proof}[Proof of Main Result A (Theorem \ref{thm:mainA})]
    Suppose $\Phi$ is a $d\times n$ ETF in real symplectic space and set $G=\frac{1}{|\Phi^\dagger\Phi_{12}|}\Phi^\dagger\Phi$. 
    We split into two cases. 
    
    For the first case, suppose that $d\equiv 2$ mod 4 and that $n=d$.
    Then $G$ is skew symmetric, has full rank, and off-diagonal entries $\pm 1$.
    By Proposition \ref{prop:cequals}, $G$ satisfies $-GG^\top = G^2 = (d-1)I_d$, implying that $G$ is a skew conference matrix. 
    Since $d\equiv 2$ mod $4$, this is only possible if $d=2$. 
    Therefore for all $d>2$, where $d\equiv 2$ mod 4, $n$ must be $d+1$, while when $d=2$, we may have $n\in \{d,d+1\}$.

    For the second case, suppose that $d\equiv 0$ mod $4$ and that $n=d+1$, so that in particular $n\equiv 1$ mod 4. 
    Consider $G$ as the Seidel adjacency matrix of some $n$-tournament $T$. 
    Recall (see~\cite{horn2012matrix}) that the characteristic polynomial $c_G(t)$ of $G$ can be written as 
    \[
    c_G(t) = t^n + \sum_{j=1}^n E_jt^{n-j} \qquad \text{where} \, E_j = (-1)^j\sum(j\times j\text{ principal minors of } G).
    \]
    Let $\delta_T$ denote the number of diamonds in $T$. 
    The number of $4\times 4$ principal minors of $G$ is ${n\choose 4}$, and each $4\times 4$ principal minor is either $9$ (if the corresponding subgraph is a diamond) or $1$ (otherwise).
    Thus 
    \[
    E_4 = 9\delta_T + 1\left({n\choose 4} - \delta_T\right) = 8\delta_T + {n\choose 4}.
    \]
    Since $\Phi$ is tight, Proposition \ref{prop:cequals} together with the diagonalizability of $G$ imply that the characteristic polynomial of $G$ is given by
    \[
    c_G(t) = t(t^2+n)^m \qquad \text{where } m = \frac{n-1}{2}.
    \]
    Applying the binomial theorem, we have
    \[
    c_G(t) = \sum_{k=0}^m {m\choose k} t^{2m-2k+1}n^k.
    \]
    In particular, the coefficient of the $t^{2m-3}$ term gives the sum of all $4\times 4$ principal minors:
    \[
    {m\choose 2}n^2 = \frac{(n-1)(n-3)n^2}{8} = 8\delta_T + {n\choose 4} .    
    \]
    Solving for $\delta_T$ yields
    \[
    \delta_T = \frac{1}{96}n(n-1)(n-3)(n+1).
    \]
    By the negation of Lemma \ref{lem:not1mod4}, it must be that $n\not\equiv 1$ mod 4, a contradiction.
\end{proof}

\begin{proof}[Proof of Main Result B (Theorem \ref{thm:equiv})]
    To start, recall that Theorem \ref{thm:confgivesetf} proves $(a)$ and the reverse direction of $(b)$.
    Therefore, we prove the forward direction of $(b)$. 
    
    Suppose that $\Phi$ is a $d \times n$ ETF in $\R_{\mathcal{S}}^d$ with $n=d+1$ and let $G$ be its Gram matrix, where without loss we have $|G_{ij}| =1$ for $i\neq j$.
    By Theorem \ref{thm:mainA}, we have $n\equiv 3$ mod 4.
    Note that $\dim\ker G=1$ as $\ker \Phi = \ker G$.
    Let $x\in \ker G$ have $x^\top x=n$ and note that $\ker G \perp \operatorname{im}G$ since $\operatorname{im}G^* = \operatorname{im}G$. 
    Proposition \ref{prop:cequals} implies that the matrix $\frac{1}{n}GG^\top$ is a Hermitian idempotent hence gives Euclidean projection onto $\operatorname{im}GG^\top = \operatorname{im}G$.
    Similarly, $\frac{1}{n}xx^\top$ gives projection onto $\ker G$. 
    Therefore
    \[
    xx^\top + GG^\top = nI_n.
    \]
    Now, put $C = \left[\begin{array}{rr}
        0 & x^\top \\
        -x & G
    \end{array}\right]$ and observe
    \[
    CC^\top = \left[\begin{array}{rr}
        0 & x^\top \\
        -x & G
    \end{array}\right]\left[\begin{array}{rr}
        0 & -x^\top \\
        x & G^\top
    \end{array}\right] = \left[\begin{array}{rr}
        n & (Gx)^\top \\
        Gx & xx^\top + GG^\top
    \end{array}\right] = nI_{n+1}.
    \]
    It remains to show that the off-diagonal entries of $C$ are $\pm 1$, and so it is enough to show that the entries of $x$ are $\pm 1$.
    
    Consider $G$ as the Seidel adjacency matrix of some $n$-tournament $T$. 
    Following the same steps in the proof of Theorem \ref{thm:mainA} above, we have that
    \begin{align*}
        \delta_T = \frac{1}{96}n(n-1)(n-3)(n+1).
    \end{align*}
    Proposition $\ref{thm:belkouchethm}$ implies that the tournament $T$ with Seidel adjacency matrix $G$ is switching equivalent to a doubly regular tournament $T'$.
    As $\mathbf{1}$ spans the kernel of the Seidel adjacency matrix of $T'$, we know that $\ker G$ is spanned by a flat vector, or one where all entries have uniform modulus.
    Since $\|x\|^2 = n$, this means that $x$ has entries $\pm 1$. 
    Thus the matrix $H = C+I_{n+1}$ is a skew Hadamard matrix of order $n+1 = d+2$.
\end{proof}

The key to the proof of Theorem \ref{thm:equiv} is that the underlying combinatorial object behind a $d\times n$ ETF in real symplectic space is a tournament that saturates the upper bound on $\delta_T$.
While we have made this clear when $n=d+1$, it is also true when $n=d$.
For an arbitrary $n$-tournament $T$ with Seidel adjacency matrix $S$, the quantity $\delta_T$ is given by (see Proposition 3 in ~\cite{belkouche2020matricial})
\[
\delta_T = \frac{1}{96}n^2(n-1)(n-2) - \sum_{i<j} (S^2)_{ij}.
\]
In particular, $\delta_T$ is maximized if $C^2 = -CC^T = (1-n)I_n$, that is, if $C$ is a skew conference matrix.

\section{Doubling Construction}

Given that the Gram matrices of $d\times d$ ETFs in real symplectic space are in bijection with skew Hadamard matrices, it is natural to ask how operations involving the latter materialize as operations involving the former.
In particular, we turn our interest towards the operations used to construct skew Hadamard matrices.

Many efforts have been made toward finding skew Hadamard matrices, and these efforts have resulted in several useful constructions (see~\cite{koukouvinos2008skew} for a survey).
One such construction, often called doubling, is a specific case of Lemma 1 in~\cite{Williamson1944Hadamard} made explicit in~\cite{Wallis_1971}.
It is straightforward: If $H$ is a skew Hadamard matrix of order $d$, then
\[
H' = \begin{bmatrix}
    H & H \\
    H-2I_d & -H+2I_d
\end{bmatrix}
\]
is a skew Hadamard matrix of order $2d$. 

Consequently, if $H=C+I_d$, then the skew conference matrix $C$ is the Gram matrix of a $d\times d$ ETF $\Phi$ in $\R_\mathcal{S}^d$, and 
\[
C' = \begin{bmatrix}
    C & C+I_d \\
    C-I_d & -C
\end{bmatrix}
\]
is the Gram matrix of a $2d\times 2d$ equiangular tight frame $F$ in $\R_{\mathcal{S}}^{2d}$.
In fact, $F$ can be written in terms of $\Phi$, as we show in the following theorem.

\begin{theorem}\label{thm:sdoubling}
     Let $\Phi$ be a $d\times d$ ETF in real symplectic space with $|\Phi^\dagger\Phi_{ij}|=\mu$ for $i\neq j$ and let $B:\R_\mathcal{S}^d\to\R_\mathcal{S}^d$ be any transformation satisfying $B^\top\Omega B = -\Omega$.
     Then
     \[
     \begin{bmatrix} \Phi^\dagger\Phi & \Phi^\dagger\Phi+\mu I_d \\ \Phi^\dagger\Phi-\mu I_d & -\Phi^\dagger\Phi \end{bmatrix}
     \]
     has off-diagonal entries $\pm \mu$ and is the Gram matrix for the $2d\times 2d$ ETF in real symplectic space given by
     \[
     F = \left[\begin{array}{cc} a\mu^{-1}\Phi\Phi^\dagger\Phi & b\Phi \\
     yB\Phi & zB\Phi\end{array}\right],
     \]
     where 
     \[
     a = \left(\dfrac{(4(d-1)^2+1)^{1/2}+2d-3}{2(d-1)^2}\right)^{1/2},
     \] 
     \[
     b = \frac{1}{a(d-1)}, \qquad y=-\frac{a(d-1)}{(a^2(d-1)^2+1)^{1/2}}, \qquad z = \frac{(a^2(d-1)^2+1)^{1/2}}{a(d-1)}. 
     \]
 \end{theorem}
 Note that the set of all transformations $B$ satisfying $B^\top\Omega B = -\Omega$ is not empty. 
 For example, let $A_1,\dots,A_{d/2}\in \R^{2\times 2}$ satisfy $\det(A_i)=-1$. 
 If $B$ is given by the matrix $\bigoplus_{i=1}^{d/2} A_i$, then $B$ satisfies $B^\top\Omega B = -\Omega$.
 \begin{proof}
     Put $\Psi = \mu^{-1/2}\Phi$ and $C=\Psi^\dagger\Psi$, so that $|C_{ij}| = 1$ for $i\neq j$ and by Theorem \ref{thm:confgivesetf} we have that $C$ is a skew conference matrix. 
     Since
     \[
    G:=\begin{bmatrix} \Phi^\dagger\Phi & \Phi^\dagger\Phi+\mu I_d \\ \Phi^\dagger\Phi-\mu I_d & -\Phi^\dagger\Phi \end{bmatrix} = \mu\begin{bmatrix}
    C & C+I_d \\
    C-I_d & -C
    \end{bmatrix},
    \]
     it follows from the doubling construction mentioned above that $G$ is the Gram matrix of a $2d\times 2d$ ETF in real symplectic space.
     It remains to show that $G= F^\dagger F$, where $F$ is as described in the theorem statement.
     
     For sake of computing the adjoint, we identify $F$ with the linear transformation $F:\R_\mathcal{E}^{2d}\to\R_{\mathcal{S}}^{2d}$, and by assumption $B\Phi$ is a map from $\R_\mathcal{E}^d$ to $\R_\mathcal{S}^d$.
     Thus
     \[
     F^\dagger = \left[\begin{array}{cc} a\mu^{-1}\Phi^\top \Omega^\top\Phi\Phi^\top & y\Phi^\top B^\top \\
     b\Phi^\top & z\Phi^\top B^\top\end{array}\right]\begin{bmatrix}
         \Omega & 0 \\ 0 &\Omega
     \end{bmatrix} = \left[\begin{array}{cc}-a\mu^{-1}\Phi^\dagger\Phi\Phi^\dagger & y(B\Phi)^\dagger \\
     b\Phi^\dagger & z(B\Phi)^\dagger\end{array}\right]
     \]
     Computing, we have
     \begin{align*}
         F^\dagger F &= \left[\begin{array}{cc}-a\mu^{-1}\Phi^\dagger\Phi\Phi^\dagger & y(B\Phi)^\dagger \\
         b\Phi^\dagger & z(B\Phi)^\dagger\end{array}\right]\left[\begin{array}{cc} a\mu^{-1}\Phi\Phi^\dagger\Phi & b\Phi \\
         y B\Phi & z B\Phi\end{array}\right] \\
         &=\mu^{1/2}\left[\begin{array}{cc}-a\mu^{-3/2}\Phi^\dagger\Phi\Phi^\dagger & y\mu^{-1/2}(B\Phi)^\dagger \\
         b\mu^{-1/2}\Phi^\dagger & z\mu^{-1/2}(B\Phi)^\dagger\end{array}\right]\mu^{1/2}\left[\begin{array}{cc} a\mu^{-3/2}\Phi\Phi^\dagger\Phi & b\mu^{-1/2}\Phi \\
         y \mu^{-1/2}B\Phi & z \mu^{-1/2}B\Phi\end{array}\right] \\
         &= \mu\left[\begin{array}{cc}-a\Psi^\dagger\Psi\Psi^\dagger & y\Psi^\top B^\top\Omega \\
         b\Psi^\dagger & z\Psi^\top B^\top\Omega\end{array}\right]\left[\begin{array}{cc} a\Psi\Psi^\dagger\Psi & b\Psi \\
         yB\Psi & zB\Psi\end{array}\right] \\
         &= \mu\left[\begin{array}{cc} 
         -a^2(\Psi^\dagger\Psi)^3+y^2\Psi^\top B^\top \Omega B \Psi & -ab(\Psi^\dagger \Psi)^2 +yz\Psi^\top B^\top \Omega B \Psi \\
         ba(\Psi^\dagger\Psi)^2 + yz \Psi^\top B^\top \Omega B \Psi & b^2 \Psi^\dagger\Psi + z^2\Psi^\top B^\top \Omega B \Psi
         \end{array}\right] \\
         &= \mu\left[\begin{array}{cc}
         -a^2(\Psi^\dagger\Psi)^3-y^2\Psi^\dagger\Psi & -ab(\Psi^\dagger\Psi)^2 - yz\Psi^\dagger\Psi \\
         ba(\Psi^\dagger\Psi)^2-yz\Psi^\dagger\Psi & b^2\Psi^\dagger\Psi -z^2\Psi^\dagger\Psi
         \end{array}\right] \\
         &= \mu\left[\begin{array}
         {cc}
         -a^2C^3-y^2C & -abC^2-yzC \\
         baC^2-yzC & (b^2-z^2)C
         \end{array}\right] \\
         &= \mu\left[\begin{array}{cc}
         (a^2(d-1)-y^2)C & -yzC + ab(d-1)I_d \\
         -yzC-ab(d-1)I_d & (b^2-z^2)C
         \end{array}\right].
     \end{align*}
     
     The result follows from the identities
     \begin{align*}
         a^2(d-1)-y^2 &= 1, & ab &= \frac{1}{d-1}, \\
         yz &= -1, & b^2-z^2 &= -1.
     \end{align*}
      Of them, the first is a straightforward (although monotonous) calculation, while the rest are immediate.
 \end{proof}

\section{Connection to Complex ETFs}
Recall from Example \ref{ex:complexsympspace} that the imaginary part of the Hermitian product on $\C^d$ is a symplectic form.
Thus in a sense, symplectic forms are `shadows' cast by the Hermitian inner product.
In a similar way, ETFs in real symplectic space are `shadows' cast by certain complex ETFs. 
Given a complex matrix
    \[\Psi = \begin{bmatrix}
        - &\psi_1^* & - \\
        & \vdots & \\
        - & \psi_d^* & -
    \end{bmatrix} \in \C^{d\times n}, \quad \text{ define } \Psi_\diamond := \begin{bmatrix}
        - &\operatorname{Re}\psi_1^* & - \\
        - & \operatorname{Im}\psi_1^* & - \\
        & \vdots & \\
        - & \operatorname{Re}\psi_d^* & - \\
        - & \operatorname{Im}\psi_d^* & -
    \end{bmatrix}\in \R^{2d\times n}.\]
     Essentially, $\Psi_\diamond$ takes the complex rows of $\Psi$ and splits them into real and imaginary parts, stacking them on top of each other.
     We identify $\Psi_\diamond$ with the linear transformation $\Psi_\diamond:\R_{\mathcal{E}}^n\to \R_{\mathcal{S}}^{2d}$, so that $\Psi_\diamond^\dagger = \Psi_\diamond^\top \Omega$.
    
    Given $\Psi\in \C^{d\times n}$, we have $\operatorname{Im}(\Psi^*\Psi) = \Psi_\diamond^\dagger\Psi_\diamond^{\phantom{\dagger}}$. 
    Indeed, label the rows of $\Psi$ as $\psi_i^*$. 
    Then
    \[
        \operatorname{Im}\Psi^*\Psi =\sum_{i=1}^d \operatorname{Im}\psi_i\psi_i^* 
        = \sum_{i=1}^d\operatorname{Re}\psi_i\operatorname{Im}\psi_i^*+\operatorname{Im}\psi_i\operatorname{Re}\psi_i^*. 
    \]
    Meanwhile,
    \[
    \Psi_\diamond^\dagger\Psi_\diamond^{\phantom{\dagger}} = \sum_{i=1}^d(\operatorname{Re}\psi_i^*)^\top\operatorname{Im}\psi_i^*-(\operatorname{Im}\psi_i^*)^\top\operatorname{Re}\psi_i^*=\sum_{i=1}^d \operatorname{Re}\psi_i\operatorname{Im}\psi_i^*+\operatorname{Im}\psi_i\operatorname{Re}\psi_i^*.
    \]

Theorem \ref{thm:symplecticfromcomplex} below shows that given any ETF $\Phi$ in real symplectic space, there exists a complex ETF $\Psi$ such that $\operatorname{Im}(\Psi^*\Psi) = \Phi^\dagger\Phi$, up to a scaling. 
Moreover, suppose $\Psi$ is a $\frac{d}{2}\times n$ complex ETF, where $n\in \{d,d+1\}$, such that $|(\Psi^*\Psi)_{ij}|=\mu$ for $i\neq j$.
If the signature matrix $Q=\frac{1}{\mu}(\Psi^*\Psi-I_d)$ of $\Psi$ has specific entries, Theorem \ref{thm:symplecticfromcomplex} says that $\operatorname{Im}(Q)$ is the Gram matrix of an ETF in real symplectic space.
To prove this theorem, we use the following information about the signature matrix of a complex ETF.
\begin{proposition}{\cite{HP:04}}\label{prop:signature}
Let $Q\in \C^{n\times n}$ be self-adjoint with $Q_{ii} = 0$ for all $i$ and $|Q_{ij}| = 1$ for all $i\neq j$, and let $d$ be a positive integer. Then the following are equivalent:
\begin{itemize}
    \item[(i)] $Q$ is the signature matrix of a $d\times n$ complex ETF,
    \item[(ii)] $Q^2 = cQ+(n-1)I_n$, where $c = (n-2d)\sqrt{\frac{n-1}{d(n-d)}}$.
\end{itemize}
\end{proposition}

\begin{theorem}\label{thm:symplecticfromcomplex}
Let $\Phi$ be a $d\times n$ frame in real symplectic space, where $n\in \{d,d+1\}$.
    \begin{enumerate} 
        \item[(a)] If $n=d$, then $\Phi$ is an ETF if and only if there exists a scalar $\alpha>0$ and a $\frac{d}{2}\times d$ complex ETF $\Psi$ with $\operatorname{Im}\Psi^*\Psi = \alpha \Phi^\dagger\Phi$, such that the signature matrix of $\Psi$ has entries only in $\{0,\pm i\}$. 
        \item[(b)] If $n=d+1$, then $\Phi$ is an ETF if and only if there exists a scalar $\alpha>0$ and a $\frac{d}{2}\times (d+1)$ complex ETF $\Psi$ with $\operatorname{Im}\Psi^*\Psi = \alpha \Phi^\dagger\Phi$, such that the signature matrix of $\Psi$ has entries only in $\{0,\beta,\overline{\beta}\}$ where 
        \[\beta = -\frac{1}{d+2}+i\sqrt{1 - \frac{1}{d+2}}.\]. 
    \end{enumerate}
\end{theorem}
\begin{proof}
    First, we prove $(a)$.
    Suppose that $n=d$ and that $\Phi$ is an ETF in real symplectic space. 
    Set $\alpha = (|(\Phi^\dagger\Phi)_{12}|\sqrt{d-1})^{-1}$ and $C = \frac{1}{|(\Phi^\dagger\Phi)_{12}|}\Phi^\dagger\Phi$.
    By Theorem \ref{thm:confgivesetf}, $C$ is a skew conference matrix.
    Therefore, the matrix $G = I_d+i(d-1)^{-1/2}C$ is the Gram matrix of a $\frac{d}{2}\times d$ complex ETF $\Psi$ (see~\cite{delsarte1991bounds}) with purely imaginary signature matrix $iC$. 
    Moreover, $\Psi$ satisfies 
    \[
    \operatorname{Im}\Psi^*\Psi = \operatorname{Im}G=(d-1)^{-1/2}C = \alpha\Phi^\dagger\Phi,
    \]
    as desired.

    In the reverse direction, let $Q$ be the purely imaginary signature matrix of $\Psi$. 
    Then $$-iQ = \operatorname{Im}(Q) = \frac{1}{|(\Psi^*\Psi)_{12}|}\operatorname{Im}(\Psi^*\Psi)= \frac{\alpha}{|(\Psi^*\Psi)_{12}|}\Phi^\dagger\Phi.$$
    It follows that $Q$ and $\operatorname{Im}Q$ are skew symmetric.
    By Proposition \ref{prop:signature}, we see that $Q^2 = (d-1)I_d$. 
    Thus  
    \[
    \operatorname{Im}(Q)(\operatorname{Im}(Q))^\top = (-iQ)(-iQ)^\top = -QQ^\top = Q^2 = (d-1)I_d.
    \]
    Because $Q_{ij}\in\{\pm i\}$ when $i\neq j$, it follows from the equality above that $\operatorname{Im}(Q)=\frac{1}{|(\Psi^*\Psi)_{12}|}\alpha\Phi^\dagger\Phi$ is a skew conference matrix.
    By Theorem 23, $\Phi$ is an ETF in real symplectic space.

    For $(b)$, suppose that $n=d+1$ and that $\Phi$ is an ETF in real symplectic space.  
    By the proof of Theorem \ref{thm:equiv}, there exists a diagonal $\pm1$-matrix $D$ so that $K:=\frac{1}{|(\Phi^\dagger\Phi)_{12}|} D\Phi^\dagger\Phi D$ is the core of a normalized skew conference matrix.
    Decompose $K$ as $K=A-A^\top$, where $A$ is a $(0,1)$-matrix of order $d+1$.
    We have (see~\cite{fallon2023optimal,strohmer2008note}) that $Q = \beta A+\overline{\beta}A^\top$ is the signature matrix of a $\frac{d}{2}\times (d+1)$ complex ETF. 
    From Proposition \ref{prop:signature}, we know that $Q^2 = \frac{2}{\sqrt{d+2}}Q+dI_{d+1}$.
    The matrix $DQD$ satisfies the same quadratic, and so $DQD$ is the signature matrix of a $\frac{d}{2}\times (d+1)$ complex ETF $\Psi$.
    Put 
    \[
    \mu = |(\Psi^*\Psi)_{12}|, \qquad  \alpha = \frac{\mu(\operatorname{Im} \beta)}{|(\Phi^\dagger\Phi)_{12}|},  
    \]
    and observe
    \begin{align*}
    \operatorname{Im}\Psi^*\Psi &= \operatorname{Im}(I_{d+1}+\mu DQD) \\
    &= \mu\operatorname{Im}(DQD) \\
    &= \mu D\operatorname{Im}(\beta A+\overline{\beta}A^\top )D  \\
    &=  \mu(\operatorname{Im} \beta) D(A-A^\top)D \\
    &= \mu(\operatorname{Im} \beta) DKD \\
    &= \alpha \Phi^\dagger\Phi,
    \end{align*}
    as desired.

    In the reverse direction, suppose that $\Psi$ is the described $\frac{d}{2}\times (d+1)$ complex ETF. 
    Then the signature matrix $Q$ of $\Psi$ can be written $Q=\beta A+\overline{\beta}A^\top$, where $A$ is a $(0,1)$-matrix of order $d+1$.
    By Proposition \ref{prop:signature}, the signature matrix satisfies
    \[ 
    Q^2 = \frac{2}{\sqrt{d+2}}Q + dI_{d+1}.
    \]
    Set $K = A-A^\top$ and let $\varepsilon = \operatorname{Re}\beta$ and $\gamma = \operatorname{Im}\beta$.
    Then 
    \begin{align*}
        Q^2 &= (\operatorname{Re}(Q)+i\operatorname{Im}(Q))^2 \\
        &= (\varepsilon(A+A^\top)+i\gamma(A-A^\top))^2 \\
        &= (\varepsilon(J_{d+1}-I_{d+1})+i\gamma K)^2 \\
        &= (\varepsilon^2(d-1)J_{d+1}+\varepsilon^2I_{d+1}-\gamma^2K)+ i\varepsilon\gamma(J_{d+1}K+KJ_{d+1}-2K)
    \end{align*}
    while also
    \[
    Q^2 = \frac{2}{\sqrt{d+2}}Q+dI_{d+1}=\frac{2}{\sqrt{d+2}}\left(\varepsilon(J_{d+1}-I_{d+1})+i\gamma K\right)+dI_{d+1}.
    \]
    Equating real parts and solving for $K^2$, we have
    \begin{align*}
        K^2 &= -(\gamma^{-2})\left(\left(\frac{2\varepsilon}{\sqrt{d+2}}-\varepsilon^2(d-1)\right)J_{d+1}+\left(d-\frac{2\varepsilon}{\sqrt{d+2}}-\varepsilon^2\right) I_{d+1}\right) \\
        &= \left(-\frac{d+2}{d+1}\right)\left(\left(-\frac{2}{d+2}-\frac{d-1}{d+2}\right)J_{d+1}+\left(d+\frac{1}{d+2}\right)I_{d+1}\right) \\
        &= J_{d+1}-(d+1)I_{d+1}.
    \end{align*}
    The eigenvalues of $K^2$ are thus $0$ with multiplicity 1 and $-(d+1)$ with multiplicity $d$. 
    It follows that the eigenvalues of $K$ are $0$ with multiplicity 1, and $\pm i\sqrt{d+1}$ each with multiplicity $d/2$ due to appearing in conjugate pairs. 
    As $K$ is skew symmetric, hence diagonalizable, its minimal polynomial splits as a product of distinct linear factors.
    In particular,
    \[
    K(K+i\sqrt{d+1}I_{d+1})(K-i\sqrt{d+1}I_{d+1}) = K^3+(d+1)K = 0.
    \]
    Thus $K$ has off-diagonal entries $\pm 1$ and satisfies $K^3 = -(d+1)K$.
    In particular, 
    \[
    K = \frac{1}{\gamma|(\Psi^*\Psi)_{12}|}\operatorname{Im}(\Psi^*\Psi)=\frac{1}{\gamma|(\Psi^*\Psi)_{12}|}\alpha\Phi^\dagger\Phi,
    \]
    and so $\Phi^\dagger\Phi$ is an ETF.
    \end{proof}

As an aside, the end of the proof of Theorem \ref{thm:symplecticfromcomplex} can be approached differently. 
Since $K$ has off-diagonal entries $\pm 1$ and is of odd order, the proof of Proposition \ref{prop:gerzon} shows that $K$ has a one-dimensional kernel. 
Because $\ker K \subseteq \ker K^2 = \operatorname{span}\{\mathbf{1}\}$, it must be that $\ker K=\operatorname{span}\{\mathbf{1}\}$. 
Building the matrix
\[
C=\begin{bmatrix}
    0 & \mathbf{1}^\top \\
    -\mathbf{1} & K
\end{bmatrix}
\]
and applying the identity $KK^\top = (d+1)I_{d+1}-J_{d+1}$, it is straightforward to see that $K$ is the core of a normalized skew conference matrix.
Alternatively yet, one can use the fact that $\mathbf{1}\in \ker K$ to deduce that $K$ is the Seidel adjacency matrix of a regular tournament $T$.
Together with the identity $K^2=J_{d+1}-(d+1)I_{d+1}$, Lemma 29 shows that $T$ is in fact doubly regular.
 
\section*{Acknowledgments}
The author is immensely grateful to Joseph W. Iverson for his guidance, thoughtfulness, and wisdom as a Ph.D. advisor. 
Additionally, the author is deeply appreciative of his comprehensive review of this document, for which his numerous suggestions have greatly improved the final product.
The author would also like to thank John Jasper and Dustin Mixon for insightful conversation.
This research was supported in part by the Johnson-Peters Summer Fellowship
from the Iowa State University Department of Mathematics.
\bigskip

\bibliographystyle{abbrv}
\bibliography{myrefs}

\newpage
\section*{Appendix}
\begin{proof}[Proof of Proposition \ref{prop:adjoint}.]
    $(a)$ Let $V^*, W^*$ denote the duals (space of linear functionals) of $V$ and $W$.
    Because $(\cdot,\cdot)_V$ is non-degenerate, the map $v\mapsto (v,\cdot)_V$ from $V$ to $V^*$ is injective, hence gives an isomorphism of $V$ with its dual, while a similar result holds for $W$ and $W^*$. 
    In particular, for any $f\in V^*$ or $g\in W^*$, they must have the form $f = (v,\cdot)_V$ or $g=(w,\cdot)_W$ for some $v\in V$ or $w\in W$, respectively.
    Non-degeneracy of the respective forms implies that the choice of $v$ and $w$ in this case uniquely determines the linear functional.

    Given $w\in W$, define a linear functional $\varphi_w$ in $V^*$ by $\varphi_w(v) = (Av,w)_W$.
    By the above, there exists a unique $y\in V$ such that $\varphi_w(v) = (Av,w)_W = (v,y)_W$ for all $v\in V$.
    Define the map $A^\dagger:W\to V$ by $A^\dagger w = y$. This map is uniquely determined and well-defined by the unique determination of $y$ from $w$, and the bilinearity of $(\cdot,\cdot)_W$ shows that $A^\dagger$ is linear, as desired.

    For $(b)$ we have 
    \[(w_1, (AB)^\dagger w_2)_W = (ABw_1,w_2)_W = (Bw_1,A^\dagger w_2)_V = (w_1,B^\dagger A^\dagger w_2)_W.\]
    We can see $(c)$ easily from $(b)$, as $I = (AA^{-1})^\dagger = (A^{-1})^\dagger A^\dagger$.
    
    For $(d)$ observe
    \[(Av,w)_W = v^\top A^\top Q_w w = v^\top Q_vQ_v^{-1}A^\top Q_ww = (v,Q_v^{-1}A^\top Q_ww)_V.\]
    The result follows from the uniqueness of $A^\dagger$.

    Finally, we prove $(e)$. 
    Without loss, suppose $V=\R_\mathcal{S}^d$ and $W=\R_\mathcal{E}^n$. Then applying $(iv)$ yields $$(A^\dagger)^\dagger = I(\Omega^{-1}A^\top I)^\top\Omega = I(IA\Omega^{-\top})\Omega = A (-\Omega^{-1}\Omega) = -A.$$
\end{proof}

\end{document}